\documentclass{article}
\usepackage{amsfonts}
\usepackage{amsmath}

\newtheorem{theorem}{Theorem}[section]

\newtheorem{corollary}[theorem]{Corollary}

\newtheorem{definition}[theorem]{Definition}
\newtheorem{example}[theorem]{Example}

\newtheorem{lemma}[theorem]{Lemma}

\newtheorem{proposition}[theorem]{Proposition}

\newenvironment{proof}[1][Proof]{\noindent\textbf{#1.} }{\ \rule{0.5em}{0.5em}}
\input{tcilatex}
\begin{document}

\title{Wiener-Hopf difference equations and semi-cardinal interpolation with 
integrable convolution kernels\thanks{This work was supported by Kuwait
University, Research Grant No.\ SM01/18.}}
\author{Aurelian Bejancu\footnote{Department of Mathematics, Kuwait University, 
PO Box 5969, Safat 13060, Kuwait. \quad 
E-mail address: aurelianbejancu@gmail.com}}
\maketitle

\begin{abstract}
Let $H\subset\mathbb{Z}^d$ be a half-space lattice, defined either relative to a fixed 
coordinate (e.g.\ $H = \mathbb{Z}^{d-1}\!\times\!\mathbb{Z}_+$), or relative to a
linear order $\preceq$ on $\mathbb{Z}^d$, i.e.\ $H = \{j\in\mathbb{Z}^d : 0\preceq j\}$. 
We consider the problem of interpolation at the points of $H$ from the space of 
series expansions in terms of the $H$-shifts of a decaying kernel $\phi$. 
Using the Wiener-Hopf factorization of the symbol for cardinal interpolation 
with $\phi$ on $\mathbb{Z}^d$, we derive some essential properties of semi-cardinal 
interpolation on $H$, such as existence and uniqueness, Lagrange series representation, 
variational characterization, and convergence to cardinal interpolation. Our main 
results prove that specific algebraic or exponential decay of the kernel $\phi$ is transferred 
to the Lagrange functions for interpolation on $H$, as in the case of cardinal interpolation. 
These results are shown to apply to a variety of examples, including the Gaussian, Mat\'{e}rn, 
generalized inverse multiquadric, box-spline, and polyharmonic B-spline kernels. \smallskip

\smallskip

\end{abstract}

\section{Introduction}

We start by reviewing some background results concerning \emph{cardinal interpolation} on the 
multi-integer lattice $\mathbb{Z}^d$, which represents a classical model in approximation theory. 

Let $\phi:\mathbb{R}^d\to\mathbb{R}$ be a continuous and symmetric kernel, suitably decaying 
for a large norm of its argument, and denote by $\mathcal{S}(\phi)$ the shift-invariant space of 
convergent infinite linear combinations of $\mathbb{Z}^d$-translates of $\phi$ with bounded 
coefficients. Cardinal interpolation with the kernel $\phi$ is the problem of constructing a function 
$s\in\mathcal{S}(\phi)$ with prescribed values at the points of  $\mathbb{Z}^d$, i.e.\ 
$s(j)=y_j$, $j\in\mathbb{Z}^d$, for a given bounded sequence $y=\{y_{j}\}_{j\in\mathbb{Z}^d}$.

This leads to a bi-infinite system of discrete convolution equations with Laurent matrix 
$L_\phi=[\phi(j-k)]_{j,k\in\mathbb{Z}^d}$. The necessary and sufficient condition for the invertibility 
of this system is that the \emph{symbol} $\sigma_\phi$, defined by the absolutely convergent series
\begin{equation}
\sigma_\phi ( \zeta ) := \sum_{k\in \mathbb{Z}^d} \phi (k) \zeta^{k},
\quad \zeta\in T^d,
\label{eq:CI-symbol}
\end{equation}
has no zero on the $d$-dimensional unit torus $T^d$ (e.g.\ Chui et al.\ 
\cite[Lemma~1.1]{cjw87}). If this condition is satisfied, then, by Wiener's lemma, $1 / \sigma_\phi$ 
also admits an absolutely convergent Fourier representation, with Fourier coefficients 
$\{a_{k}\}_{k\in\mathbb{Z}^d}$, say. In turn, these define the inverse Laurent matrix 
$L_\phi^{-1}=[a_{j-k}]_{j,k\in\mathbb{Z}^d}$, as well as the \emph{Lagrange} 
(or \emph{fundamental}) \emph{function}  
\begin{equation}
\chi ( x ) := \sum_{k\in \mathbb{Z}^d} a_k \phi (x-k) ,\quad x\in\mathbb{R}^d,
\label{eq:CI-lagrange}
\end{equation}
satisfying the interpolation conditions $\chi(0) = 1$ and $\chi(j) = 0$ for $j\in\mathbb{Z}^d\setminus\{0\}$. 
Hence, provided that $\chi$ also decays suitably, the unique solution $s$ to the cardinal interpolation 
problem admits the absolutely and uniformly convergent \emph{Lagrange representation} in terms 
of the shifts of $\chi$:
\begin{equation}
s ( x ) =\sum_{j\in\mathbb{Z}^d} y_{j} \chi ( x-j ) ,\quad x\in\mathbb{R}^d.
\label{eq:CI-scheme}
\end{equation}

Note that the localization quality of this representation is determined by the rate of dampening 
of $|\chi(x)|$ for large $\|x\|$. For instance, if $\phi$ is a univariate B-spline (Schoenberg 
\cite{schoen46,schoen73}) or a multivariate box-spline with non-vanishing symbol $\sigma_\phi$ 
(de Boor et al.\ \cite{dBHR85,dBHR93}), both of whom are compactly 
supported kernels, or if $\phi$ is the Gaussian kernel (Sivakumar \cite{siva97}), then the 
Lagrange function $\chi$ decays exponentially. More generally, it follows via analyticity arguments  
that, if $|\phi(x)| = O (e^{-\alpha\|x\|})$, then $|\chi(x)| = O(e^{-\beta\|x\|})$, for some $\beta\in(0,\alpha)$ 
(see section 2). On the other hand, if $\phi$ has an algebraic decay with rate $\alpha>d$, i.e.\ 
$|\phi(x)| = O ((1+\|x\|)^{-\alpha})$, then $\chi$ decays with exactly the same rate 
$|\chi(x)| = O((1+\|x\|)^{-\alpha})$, as proved by Bacchelli et al.\ \cite{bbrv05}, using 
a well-known result of Jaffard \cite{sj90}.

Recently, Fageot et al.\ \cite{fuw19} considered a different type of decay, expressed as a  
weighted $\ell_1$-type condition, for a weight function $w$:
\begin{equation}
\sup_{x\in [0,1]^d} \sum_{k\in \mathbb{Z}^d} | \phi (x-k) | w(k) < \infty.
\label{eq:weighted-decay}
\end{equation}
Assuming that $w$ fulfills certain admissibility conditions and the cardinal symbol $\sigma_\phi$ 
does not vanish on the unit torus, they proved \cite[Proposition~27]{fuw19} that property 
(\ref{eq:weighted-decay}) is inherited by the corresponding Lagrange function $\chi$. 

The above results on transferring the decay of the kernel $\phi$ to its associated Lagrange function 
$\chi$ for cardinal interpolation rely crucially on the fact that Wiener's lemma holds not only in the 
classical Wiener algebra of symbols with absolutely convergent Fourier expansion, but also in 
subalgebras of symbols defined by stronger localization properties of Fourier coefficients, such 
as exponential or algebraic decay, or finiteness of some weighted $\ell_1$-norm. This fundamental 
principle of localization transfer under inversion, or inverse-closedness, has been extended in 
recent decades far beyond 
commutative algebras of symbols (or, equivalently, of Laurent matrices), to large classes of 
non-commutative algebras of matrices, integral operators, and pseudo-differential operators with 
off-diagonal decay (e.g.\ Gr\"{o}chenig \cite{gro10} and references therein).

Our paper deals with a new application of this principle to approximation theory, in the non-commutative 
setting of semi-infinite Toeplitz matrices generated by the problem of semi-cardinal kernel interpolation 
on a half-space lattice $H\subset\mathbb{Z}^d$. We will consider two types of half-space lattices, 
determined either relative to a fixed coordinate (e.g.\ $H = \mathbb{Z}^{d-1}\!\times\!\mathbb{Z}_+$), 
or relative to a linear order $\preceq$ on $\mathbb{Z}^d$, such that $\mathbb{Z}^d$ becomes an 
(additive) ordered group and $H = \{j\in\mathbb{Z}^d : 0\preceq j\}=:\mathbb{Z}^d_{\preceq,+}$. In both 
cases, we formulate the problem of \emph{semi-cardinal interpolation} on $H$ with the kernel $\phi$, 
as follows: 

Given a bounded sequence of real values $y=\{y_j\}_{j\in H}$, determine a bounded sequence of 
coefficients $c=\{c_k\}_{k\in H}$, such that the function 
\begin{equation}
s ( x ) = \sum_{k\in H} c_k \phi (x-k) ,\quad x\in\mathbb{R}^d,
\label{eq:SCI-repres}
\end{equation}
satisfies the interpolation conditions $s(j) = y_j$, for all $j\in H$. 

This problem is equivalent to solving the multi-index Wiener-Hopf system of difference 
equations
\begin{equation}
\sum_{k\in H} c_k \phi (j-k) = y_j,\quad j\in H,
\label{eq:SCI-system}
\end{equation}
with the semi-infinite Toeplitz matrix $T_\phi=[\phi(j-k)]_{j,k\in H}$. Based on methods developed 
in the case $d=1$ by Krein \cite{mgk58} and by Calderon et al.\ \cite{csw59}, the necessary and 
sufficient conditions for the invertibility of such systems, as well as the explicit form of their solution 
in the multi-index case ($d>1$), were given, for $H = \mathbb{Z}^{d-1}\!\times\!\mathbb{Z}_+$, 
by Goldenstein and Gohberg \cite{lsg64,gg60}, and, for $H=\mathbb{Z}^d_{\preceq,+}$, by 
van der Mee et al.\ \cite{vsr03} (see also B\"{o}tcher and Silbermann \cite{bs90}, Gohberg and 
Feldman \cite{gf74}). However, these studies did not consider 
whether any specific off-diagonal decay of the Toeplitz matrix $T_\phi$, expressed as 
decay of the sequence $\{\phi(k)\}_{k\in\mathbb{Z}^d}$, may carry over to the entries of the 
inverse matrix $T_\phi^{-1}$.

Using the explicit form of $T_\phi^{-1}$, our main results, Theorems~\ref{th:SCI-decay-symb}
and \ref{th:SCI-decay-Lag}, prove that, in analogy to cardinal interpolation, the algebraic 
or exponential decay of the kernel $\phi$ is converted into corresponding off-diagonal decay of 
the inverse matrix and further transferred to the Lagrange functions 
for semi-cardinal interpolation. We also establish, under minimal assumptions on $\phi$, 
some fundamental properties of semi-cardinal kernel interpolation that follow from the 
Schur condition satisfied by the matrix $T_\phi^{-1}$. 

These results are obtained in section 3, by constructing the Lagrange scheme for semi-cardinal 
interpolation on $H$. Whereas the cardinal interpolation scheme (\ref{eq:CI-scheme}) uses the 
shifts of a single function $\chi$, this is no longer valid for the semi-cardinal Lagrange functions. 
Thus, for each $j\in H$, we construct a separate Lagrange function $\chi_j$, whose coefficient 
sequence in representation (\ref{eq:SCI-repres}) is defined by means of the factors of the so-called 
\emph{Wiener-Hopf} (or \emph{canonical}) \emph{factorization} of the cardinal inverse symbol 
$\sigma_\phi^{-1}$. This factorization procedure leads to a direct method, via discrete convolution 
estimates, for proving the required decay properties 
for the family of Lagrange functions $\chi_j$, $j\in H$, with constants that are independent of $j$. 
Specifically, the transfer of exponential decay to the Wiener-Hopf factors employs 
analyticity arguments, while the similar property for algebraic decay is based on a Wiener-L\'{e}vy 
lemma for Jaffard's algebra of symbols (Lemma~\ref{le:wiener-levy}). 
A second, indirect operatorial method of proof is also provided, by applying Jaffard's 
well-known inverse-closedness results \cite{sj90} for infinite matrices with off-diagonal decay.

We remark that, for $d>1$, interpolation on $H=\mathbb{Z}^{d-1}\!\times\!\mathbb{Z}_+$ from 
the shifts of a decaying kernel has previously been considered by Bejancu and Sabin 
\cite{ab06,bs05} only for $d=2$ 
and $\phi = M_{2,2,2}$, the compactly supported three-direction box-spline of multiplicity $2$ 
in each direction. In these references, the representation (\ref{eq:SCI-repres}) is augmented 
with an extra layer of shifts of the kernel, which allows the formulation of boundary conditions 
in terms of kernel coefficients. As detailed in Example~\ref{ex:box}, the approximation 
results obtained for the stationary schemes treated in \cite{ab06,bs05} demonstrate that 
semi-cardinal interpolation can be an effective model for the analysis of boundary effects 
of kernel interpolation schemes under scaling. On the other hand, kernel interpolation 
on a half-space lattice $H=\mathbb{Z}^d_{\preceq,+}$ induced by a linear order on 
$\mathbb{Z}^d$ has not been studied before. As shown in Corollary~\ref{cor:SCI-operator}, 
the order structure leads to a Cholesky factorization of the inverse matrix $T_\phi^{-1}$, 
via the notion of a multi-index triangular matrix. 

The paper is laid out as follows. In section 2, we review material pertaining to cardinal 
interpolation, including basic assumptions on the kernel $\phi$, as well as some results and 
concepts that are used in the next sections. The main results on semi-cardinal interpolation 
appear in section 3, where we treat both types of half-space lattices in parallel, specifying 
their differences en route. Section 4 obtains two types of results, the first one dealing with 
the approximation relationship between the cardinal and semi-cardinal interpolation schemes, 
while the second proving a variational characterization that extends a classical spline 
property. In the last section, our results are applied to five classes of examples: the Gaussian, 
Mat\'{e}rn, generalized inverse multiquadric, box-spline, and polyharmonic B-spline kernels, leading 
not only to new results, but also to some improvements and extensions of current literature.
\medskip

\noindent
\textbf{Notation remarks.} 
In the sequel, we use $\ell_p := \ell_p (\mathbb{Z}^d)$. Also,
$T^d$ denotes the unit $d$-dimensional torus
$T^d := \{ \zeta=(\zeta_1,\dots,\zeta_d) \in \mathbb{C}^d : |\zeta_1|=\ldots=|\zeta_d|=1 \}$,
and $W:=W(T^d)$ is the Wiener algebra of continuous functions $u$ on 
$T^d$ with Fourier coefficients $\{\widehat{u}_k\}_{k\in\mathbb{Z}^d}\in\ell_1$. The norm 
of $u\in W$ is $\| u \|_W = \sum_{k\in\mathbb{Z}^d} | \widehat{u}_k |$. For 
$\zeta=(\zeta_1,\dots,\zeta_d) \in \mathbb{C}^d$ and $k=(k_1,\dots,k_d)\in\mathbb{Z}^d$, 
we use $\zeta^k = \zeta_1^{k_1}\cdots\zeta_d^{k_d}$.

\section{Cardinal interpolation on the lattice $\mathbb{Z}^d$}
\addtocounter{equation}{-6}

This section has two parts: in the first, we review the basic questions on existence, uniqueness, 
and Lagrange representation for solutions of the cardinal interpolation problem, under mild 
assumptions on the kernel $\phi$. The second part imposes specific algebraic or exponential decay 
conditions on $\phi$ and looks at how these are transferred to the Lagrange function $\chi$ for 
cardinal interpolation. The presentation is intended as a blueprint for the new results on 
semi-cardinal interpolation provided in the next section.

\subsection{The cardinal interpolation scheme}

Let $\phi:\mathbb{R}^d\to\mathbb{R}$ be a continuous and \emph{symmetric} function, 
i.e. $\phi (-x) = \phi (x)$,  $x\in\mathbb{R}^d$.
The main assumption we make on $\phi$ in this subsection is the condition
\begin{equation}
|\phi|_\infty := \sup_{x\in [0,1]^d} \sum_{k\in \mathbb{Z}^d} | \phi (x-k) | < \infty,
\label{eq:minimal}
\end{equation}
which ensures that the shift-invariant space $\mathcal{S}(\phi)$ described in section~1 
consists of well-defined continuous functions $s$ of the form
\begin{equation}
s ( x ) = \sum_{k\in \mathbb{Z}^d} c_k \phi (x-k) ,\quad x\in\mathbb{R}^d,
\label{eq:CI-repres}
\end{equation}
for an arbitrary bounded sequence $c=\{c_k\}_{k\in\mathbb{Z}^d}$. Indeed, the last series is 
then absolutely and uniformly convergent on $\mathbb{R}^d$. 
Note that (\ref{eq:minimal}) also implies 
$\phi\in L_1 (\mathbb{R}^d)\cap L_\infty (\mathbb{R}^d)$.

Assumption (\ref{eq:minimal}) and related conditions using $L_p$-norms on 
the cube $[0,1]^d$, have been employed by Jia and Micchelli \cite{jm91}, in their work on 
multiresolution analysis, and are also known to provide a natural setting for principal shift-invariant 
theory, e.g.\ Johnson \cite{mjj97}. In the context of cardinal interpolation, Fageot 
et al.\ \cite{fuw19} have recently used the stronger condition (\ref{eq:weighted-decay}), 
which they relate to the general concept of Wiener amalgam spaces. 

As stated in the Introduction, for a bounded sequence $y=\{y_{j}\}_{j\in\mathbb{Z}^d}$ of real values, 
the problem of cardinal interpolation with $\phi$ is to find a function of the form (\ref{eq:CI-repres}), 
such that
\begin{equation}
s ( j ) =y_{j},\qquad \forall \ j\in \mathbb{Z}^d.
\label{eq:CI-problem}
\end{equation}
This amounts to the infinite system of discrete 
convolution equations for the coefficient sequence $c$:
\begin{equation}
\sum_{k\in \mathbb{Z}^d} c_k \phi (j-k) = y_j,\quad j\in\mathbb{Z}^d.
\label{eq:CI-system}
\end{equation}
We use the notation $L_\phi = [\phi(j-k)]_{j,k\in\mathbb{Z}^d}$ for both the bi-infinite multi-index 
Laurent matrix of this system and the corresponding convolution operator such that, for a bounded 
sequence $c$, $L_\phi c$ is the sequence whose $j$-th component is given by the 
left-hand side of (\ref{eq:CI-system}). Then, for any $p\in[1,\infty]$, 
$L_{\phi} \colon \ell_p  \to \ell_p $ is a bounded linear operator, 
since, by the classical Schur test (see Gr\"{o}chenig \cite{gro01}), 
\begin{equation*}
\| L_{\phi} c \|_p  \leq \| \sigma_\phi \|_W  \| c \|_p , \quad c\in \ell_p.
\end{equation*}
Recall that $\sigma_\phi$ is the symbol defined by (\ref{eq:CI-symbol}),
which belongs to the Wiener algebra $W$, under assumption (\ref{eq:minimal}). 

As mentioned already, it is well-known (e.g.\ Chui et al.\ \cite[Lemma~1.1]{cjw87}) 
that the necessary and sufficient condition for the invertibility of $L_\phi$, i.e.\  the invertibility 
of the system (\ref{eq:CI-system}), is that $\sigma_\phi ( \zeta ) \not = 0$, for all $\zeta\in T^d$. 
Under this condition, Wiener's lemma (e.g.\ Rudin \cite[p.~278]{rudin91}) ensures that the 
\emph{inverse} symbol $1/\sigma_\phi$ is also in $W$ and the solution of (\ref{eq:CI-system}) 
is given explicitly by the formula
\begin{equation}
c_k = \sum_{j\in \mathbb{Z}^d} a_{k-j} y_j,\quad k\in\mathbb{Z}^d,
\label{eq:CI-sol-coeffs}
\end{equation}
where $a=\{a_k\}_{k\in\mathbb{Z}^d}\in\ell_1$ is the sequence of Fourier 
coefficients of $1/\sigma_\phi$. 

Since the symbol $\sigma_\phi$ is continuous on $T^d$ and it takes only real values 
(due to the symmetry of $\phi$), we may replace the non-vanishing condition on $T^d$, 
without loss of generality, by the positivity condition
\begin{equation}
\sigma_\phi ( \zeta ) > 0,\quad  \zeta\in T^d.
\label{eq:pos-symbol}
\end{equation}
%

%
\begin{theorem}      \label{th:E&U-CI}
Suppose that the continuous and symmetric kernel $\phi$ satisfies 
\textup{(\ref{eq:minimal})} and \textup{(\ref{eq:pos-symbol})}. Then, 
for any sequence of data values $y\in \ell_{\infty}$, there exists a unique sequence of 
coefficients $c\in \ell_{\infty}$, such that the function \textup{(\ref{eq:CI-repres})} 
satisfies the cardinal interpolation system \textup{(\ref{eq:CI-problem})}. Specifically, 
$c$ is given by the explicit formula \textup{(\ref{eq:CI-sol-coeffs})} and the corresponding 
interpolant $s$ admits the Lagrange representation \emph{(\ref{eq:CI-scheme})}, 
where $\chi:=\chi_\phi$ is the Lagrange function \emph{(\ref{eq:CI-lagrange})} 
that interpolates the Kronecker delta sequence: 
\begin{equation}
\chi (j) = \delta_{j,0}, \quad j\in\mathbb{Z}^d.
\label{eq:CI-delta}
\end{equation}
Further, if $I_{\phi}$ denotes the interpolation operator that associates to each bounded 
sequence $y$ its unique cardinal interpolant $I_{\phi}\, y:=s$, as above, then this operator 
is bounded from $\ell_{\infty}$ to $L_{\infty} (\mathbb{R}^d)$, and its norm 
(or `Lebesgue constant') satisfies
\begin{equation}
\| I_{\phi} \|_{\infty} = |\chi|_\infty \leq \|\omega_\phi\|_W |\phi|_\infty,
\label{eq:CI-Lebesgue}
\end{equation}
where $\omega_\phi := 1/\sigma_\phi$ and the notation $|\cdot|_\infty$ is defined in the left-hand 
side of \textup{(\ref{eq:minimal})}.
\end{theorem}

\begin{proof}
The series (\ref{eq:CI-lagrange}) is absolutely and uniformly convergent on $\mathbb{R}^d$, 
since $a\in\ell_1$ and $\phi$ is bounded. Substituting $x=j\in\mathbb{Z}^d$,  
the interpolation conditions (\ref{eq:CI-delta}) follow, simply by identifying the Fourier 
coefficients of the product $\sigma_\phi^{-1} \sigma_\phi = 1$.

To prove uniqueness, consider the homogeneous system of discrete convolution 
equations associated to (\ref{eq:CI-system}):
\begin{equation*}
\sum_{k\in \mathbb{Z}^d} c_k \phi (j-k) = 0,\quad j\in\mathbb{Z}^d,
\end{equation*}
where $c$ is a bounded sequence. For every $l\in\mathbb{Z}^d$, multiplying each 
of the above equations by $a_{l-j}$ and summing over $j$, one obtains, after re-indexing,
\begin{eqnarray*}
0 & = & \sum_{j\in \mathbb{Z}^d} a_{l-j} \sum_{k\in \mathbb{Z}^d} c_k \phi (j-k)   
\ = \ \sum_{k\in \mathbb{Z}^d} c_k \sum_{j\in \mathbb{Z}^d} a_{l-j} \phi (j-k)   \\
& = & \sum_{k\in \mathbb{Z}^d} c_k \sum_{j\in \mathbb{Z}^d} a_{j} \phi (l-k-j)  
\ = \ \sum_{k\in \mathbb{Z}^d} c_k \chi (l-k) 
\ = \ \sum_{k\in \mathbb{Z}^d} c_k \delta_{l-k,0} \ = \ c_l.
\end{eqnarray*}
Changing the order of summation in the second equality is justified by the absolute 
convergence of the double sum:
\begin{eqnarray*}
\sum_{j,k\in \mathbb{Z}^d} |a_{l-j} c_k \phi (j-k)|  
& = & \sum_{j\in \mathbb{Z}^d} |a_{l-j}| \sum_{k\in \mathbb{Z}^d} |c_k| |\phi (j-k)|   \\
& \leq & \|a\|_1 \|c\|_{\infty} \sum_{k\in \mathbb{Z}^d} |\phi (k)| \ < \ \infty.
\end{eqnarray*}

Turning to the existence statement, we have, for all $x\in\mathbb{R}^d$:
\begin{eqnarray*}
\lefteqn{ \sum_{j\in \mathbb{Z}^d} y_{j} \chi (x-j)   
\ = \ \sum_{j\in \mathbb{Z}^d} y_j \sum_{k\in \mathbb{Z}^d} a_{k} \phi (x-j-k) }  \\
& = & \sum_{j\in \mathbb{Z}^d} y_j \sum_{k\in \mathbb{Z}^d} a_{k-j} \phi (x-k)  
\ = \ \sum_{k\in \mathbb{Z}^d} c_k \phi (x-k),
\end{eqnarray*}
where $c_k$ is given by (\ref{eq:CI-sol-coeffs}). The above manipulations are enabled by the 
absolute convergence of the double series:
\begin{eqnarray*}
\sum_{j,k\in \mathbb{Z}^d} |y_j a_{k-j} \phi (x-k)|  
& = & \sum_{k\in \mathbb{Z}^d} |\phi (x-k)| \sum_{j\in \mathbb{Z}^d} |y_j| |a_{k-j}|  \\
& \leq & \|y\|_{\infty} \|a\|_1 \sup_{x\in\mathbb{R}^d} \sum_{k\in \mathbb{Z}^d} |\phi (x-k)| \ < \ \infty.
\end{eqnarray*}
It follows that the function on the right-hand side of (\ref{eq:CI-scheme}), which evidently satisfies 
the cardinal interpolation conditions (\ref{eq:CI-problem}), is of the required form (\ref{eq:CI-repres}). 
Therefore the unique cardinal interpolant of this form to the data sequence $y$ is given by 
the Lagrange representation (\ref{eq:CI-scheme}).

For the last statement of the theorem, observe that, by the argument of the previous paragraph, 
the `Lebesgue function' $\Lambda := \sum_{j\in\mathbb{Z}^d} |\chi ( \cdot - j )|$ is periodic and 
continuous, being the sum of a uniformly convergent series of continuous functions. Hence, 
\begin{equation*}
\sup_{x\in\mathbb{R}^d} \Lambda (x) = \max_{x\in [0,1]^d} \Lambda (x) = |\chi|_\infty 
\leq \|\omega_\phi\|_W  |\phi|_\infty.
\end{equation*}
The equality 
$\| I_{\phi} \|_{\infty} = \sup_{x\in\mathbb{R}^d} \Lambda (x)$ follows via a standard argument 
(e.g.\ Riemenschneider and Sivakumar \cite[Theorem 3.1]{rs99b}).
\end{proof}

\begin{corollary}
Under the hypotheses of Theorem~\ref{th:E&U-CI}, for any function $s$ of the form 
\textup{(\ref{eq:CI-repres})}, where $c$ is a bounded sequence, 
the following reproduction formula holds:
\begin{equation}
s ( x ) =\sum_{j\in\mathbb{Z}^d} s(j) \chi ( x-j ) ,\quad x\in\mathbb{R}^d.
\label{eq:CI-reprod}
\end{equation}
In particular, $\phi =\sum_{j\in\mathbb{Z}^d} \phi(j) \chi ( \cdot-j )$.
\end{corollary}

\noindent
\textbf{Remarks.} 
(i) As the above proof shows, the finiteness of the Lebesgue constant (\ref{eq:CI-Lebesgue}) 
is a direct consequence of the similar condition (\ref{eq:minimal}) 
satisfied by $\phi$. This is a special case ($p=\infty$) of the $L_p$-result on periodizations 
of semi-discrete convolutions obtained by Jia and Micchelli \cite[Theorem~2.1]{jm91}. 
It is also a special case (for the trivial weight) of the recent result of Fageot et al.\ 
\cite[Lemma~26]{fuw19} concerning the transfer of weighted versions of (\ref{eq:minimal}) 
to $\chi$. The aim of the next subsection is to show that specific algebraic or exponential 
decay of $\phi$ is also inherited by the cardinal Lagrange function $\chi$.

(ii) A different consequence of \cite[Theorem~2.1]{jm91} is that the cardinal interpolation 
operator $I_{\phi}$ of Theorem~\ref{th:E&U-CI} is a bounded operator from $\ell_p$ to 
$L_p (\mathbb{R}^d)$, for $1\leq p\leq \infty$ (see also \cite[Theorem~IV.14]{dBHR93}).

\subsection{Transfer of kernel decay to Lagrange function}

Here, we consider kernels with algebraic or exponential decay and we demonstrate how such 
decay carries over to the Lagrange function for cardinal interpolation. First, we collect 
the results needed on symbols $u\in W$, whose Fourier coefficients 
$\{\widehat{u}_k\}_{k\in\mathbb{Z}^d}$ decay algebraically or exponentially. 
For $x=(x_1,\ldots,x_d)\in\mathbb{R}^d$, we use $\|x\| = (x_1^2 + \cdots + x_d^2)^{1/2}$ to 
denote the Euclidean norm and $|x| = |x_1| + \cdots + |x_d|$, the $1$-norm. 
\medskip

\noindent
\textbf{Two inverse-closed classes of symbols}

\begin{definition} 	\label{def:symbol-class}
\emph{(i)} Let $\alpha>d$. A continuous symbol $u$ on $T^d$ belongs to the class 
$J_\alpha:=J_\alpha (T^d)$ if there exists $C(u)>0$, such that 
\begin{equation}
|\widehat{u}_k| \leq C(u) (1+ \|k\|)^{-\alpha}, \quad  k\in \mathbb{Z}^d.
\label{eq:Jaffard-alg}
\end{equation}

\emph{(ii)} For $\alpha>0$, a continuous symbol $u$ on $T^d$ belongs to the class 
$\mathcal{E}_\alpha:=\mathcal{E}_\alpha (T^d)$ if there exists $C(u)>0$, such that
\begin{equation}
|\widehat{u}_k| \leq C(u) \, e^{-\alpha |k|},  \quad  k\in \mathbb{Z}^d.
\label{eq:DCP-exp1}
\end{equation}
Also, define $\mathcal{E}:= \bigcup_{\alpha>0} \mathcal{E}_\alpha$.
\end{definition}

The next lemma implies that $J_\alpha$ and $\mathcal{E}$ are subalgebras of 
the Wiener algebra $W$. Although this follows from a more general result of Jaffard 
\cite[Proposition~1]{sj90} concerning infinite matrices with off-diagonal decay, the statement 
and proof given here provide the decay constants explicitly, as needed in section~3.

%
\begin{lemma}      \label{le:DCP}
\textup{(i)} If $u,v\in J_\alpha$, then $uv\in J_\alpha$, where the corresponding constant 
in \emph{(\ref{eq:Jaffard-alg})} can be chosen as $C(uv) := 2^\alpha (C(u) \|v\|_W + C(v) \|u\|_W)$.

\textup{(ii)} If $u,v\in\mathcal{E}_\beta$, for some $\beta>0$, then $uv\in\mathcal{E}_{\beta'}$ for every 
$\beta' \in (0,\beta)$, with a corresponding constant in \emph{(\ref{eq:DCP-exp1})} given by 
$C(uv):=C(u) C(v) \coth^d ( \frac{\beta-\beta'}{2} )$. 
\end{lemma}
%
\begin{proof}
(i) This was previously proved in Bacchelli et al.\ \cite[Lemma~1]{bbrv05}, where it is formulated 
equivalently for sequences, with the discrete convolution product, rather than symbols.
For each $j\in\mathbb{Z}^d$, the $j$-th Fourier coefficient of the product $uv$ is estimated 
via a classical convolution argument:
\begin{eqnarray*}
|\widehat{(uv)}_j| & = & |\sum_{k\in\mathbb{Z}^d} \widehat{u}_{j-k} \widehat{v}_k|  \\
& \leq & \sum_{k\in\mathbb{Z}^d} | \widehat{u}_{j-k} |  | \widehat{v}_k | =: S_1 + S_2, 
\end{eqnarray*}
where $S_1$ is the part of the series corresponding to $\|k\|\leq\|j\|/2$ (hence, $\|j-k\| \geq \|j\|/2$), 
while $S_2$ corresponds to $\|k\| > \|j\|/2$. Next, we have:
\begin{eqnarray*}
S_1 & \leq & C(u) \sum_{\|k\|\leq\|j\|/2}  (1+\|j-k\|)^{-\alpha}  |\widehat{v}_k|  \\
& \leq & C(u) (1+\|j\|/2)^{-\alpha} \sum_{\|k\|\leq\|j\|/2}  | \widehat{v}_k | \  
\leq \  C(u) \|v\|_W (1+\|j\|/2)^{-\alpha} ,
\end{eqnarray*}
\begin{eqnarray*}
S_2 & \leq & C(v) \sum_{\|k\| > \|j\|/2} | \widehat{u}_{j-k} | (1+\|k\|)^{-\alpha}   \\
& \leq & C(v) (1+\|j\|/2)^{-\alpha} \sum_{\|k\| > \|j\|/2}  | \widehat{u}_{j-k} | \  
\leq \  C(v) \|u\|_W (1+\|j\|/2)^{-\alpha} .
\end{eqnarray*}
Therefore (\ref{eq:Jaffard-alg}) holds for the Fourier coefficients of $uv$, 
with the stated constant. 

(ii) For each $j\in\mathbb{Z}^d$, we have:
\begin{eqnarray*}
|\widehat{(uv)}_j| & \leq & \sum_{k\in\mathbb{Z}^d} | \widehat{u}_{j-k} |  | \widehat{v}_k |  \\
& \leq & C(u) C(v) \sum_{k\in\mathbb{Z}^d} e^{-\beta |j-k|} e^{-\beta |k|} .
\end{eqnarray*}
Since $-|j-k|\leq|k|-|j|$, the last sum is bounded above, for $\beta' \in (0,\beta)$, by 
\begin{eqnarray*}
\sum_{k\in\mathbb{Z}^d} e^{-\beta' |j-k|} e^{-\beta |k|} 
& \leq & \sum_{k\in\mathbb{Z}^d}  e^{-\beta' |j|} e^{-(\beta-\beta') |k|}  \\
& = & e^{-\beta' |j|} \left( \sum_{\nu\in\mathbb{Z}}  e^{-(\beta-\beta') |\nu|} \right)^d
\ = \  e^{-\beta' |j|} M,
\end{eqnarray*}
where 
$M := \left( \frac{1+e^{\beta'-\beta}}{1-e^{\beta'-\beta}} \right)^d = \coth^d ( \frac{\beta-\beta'}{2} )$.
Hence, (\ref{eq:DCP-exp1}) holds for the Fourier coefficients of $uv$, with the rate $\beta'$ and 
the stated constant. 
\end{proof}
\bigskip

The following result expresses the fact that $J_\alpha$ and $\mathcal{E}$ are inverse-closed 
subalgebras, i.e.\ Wiener's Lemma holds for these two algebras.

\begin{lemma}		\label{le:wiener}
\textup{(i)} If $u\in J_\alpha$ and $u$ has no zeros on $T^d$, then $1/u\in J_\alpha$.

\textup{(ii)} If $u\in\mathcal{E}$ and $u$ has no zeros on $T^d$, then $1/u\in\mathcal{E}$.
\end{lemma}

%
\begin{proof}
The proof of both parts can be deduced from the more general results of Jaffard \cite{sj90}, 
Propositions~3 and 2, concerning the class of invertible infinite matrices with off-diagonal 
algebraic or exponential decay, using the correspondence between symbols 
$u\in W$ and Laurent matrices $[ \widehat{u}_{j-k}]_{j,k\in \mathbb{Z}^d}$ generated with their 
Fourier coefficients. The proof of part (i), formulated for sequences and their discrete convolution 
product, was given by Bacchelli et al.\ \cite[Lemma~2]{bbrv05}.

For part (ii), Fageot et al.\ \cite[Theorem~13]{fuw19} gave a short alternative proof based on the 
fact that $u\in\mathcal{E}$ if and only if $u$ is real analytic on $T^d$.

Here, we present a different proof based on complex analyticity for the following version of (ii), 
adapted to our setting: if $u\in\mathcal{E}_\alpha$ and 
$u(\zeta)>0$ for all $\zeta\in T^d$, then $1/u\in\mathcal{E}_\beta$, for every $\beta \in (0,\alpha)$. 

Indeed, the exponential decay property (\ref{eq:DCP-exp1}) implies that $u$ can be extended 
to a function defined by the Fourier-Laurent series 
$u ( \zeta ) = \sum_{k\in \mathbb{Z}^d} \widehat{u}_{k} \zeta^{ k }$,
which is analytic in the neighborhood $\mathcal{A}_\alpha^d$ of the unit torus $T^d$, where 
\begin{equation*}
\mathcal{A}_\alpha:=\{z\in\mathbb{C}:e^{-\alpha}<|z|<e^{\alpha}\}.
\end{equation*}
By reducing $\alpha$, if necessary, the continuity of $u$ and its positivity on $T^d$ imply that the 
real part of $u ( \zeta )$ remains positive for all $\zeta\in\mathcal{A}_\alpha^d$, hence $u$ is 
\emph{sectorial} on $\mathcal{A}_\alpha^d$, i.e.\ the convex hull of its range does not contain 
the origin \cite[p.\ 21]{bs98}. Therefore $1/u(\zeta)$ is also analytic in $\mathcal{A}_\alpha^d$.  
In particular, for every $\beta\in(0,\alpha)$, the usual Cauchy estimates imply that the 
Fourier-Laurent coefficients of $1/u$ satisfy the exponential decay condition (\ref{eq:DCP-exp1}), 
with the rate $\beta$ in place of $\alpha$ and with the constant 
$C(1/u):=\max \{ |1/u(\zeta)| : |\zeta_p|=e^{\pm\beta}, p = 1,\ldots,d\}$. Hence, 
$1/u\in\mathcal{E}_\beta$. Note that the constant $C(1/u)$ may become arbitrarily large, 
in general, as $\beta$ approaches $\alpha$.
\end{proof}
\medskip

\noindent
\textbf{Algebraic decay transfer.} 
For the next result, we assume that $\phi$ has algebraic (or polynomial) decay, 
i.e.\ there exist constants $\alpha>d$, $C_0>0$, such that 
\begin{equation}
|\phi ( x )| \leq C_0 (1+ \|x\|)^{-\alpha}, \quad x\in \mathbb{R}^d.
\label{eq:decay-alg}
\end{equation}
Clearly, a kernel $\phi$ with this property also satisfies (\ref{eq:minimal}).

%
\begin{theorem}   \label{th:CI-decay-poly}
If the continuous symmetric kernel $\phi$ satisfies \emph{(\ref{eq:decay-alg})}, with $\alpha>d$,
and the cardinal symbol $\sigma_\phi$ satisfies the positivity condition \emph{(\ref{eq:pos-symbol})},
then both the sequence $a=\{a_k\}_{k\in\mathbb{Z}^d}$ and the associated Lagrange function 
$\chi$ defined by \emph{(\ref{eq:CI-lagrange})} decay with at least the same power as $\phi$, namely, 
there exist positive constants $C_{1,2}=C_{1,2} (\alpha,d,\phi)$, such that:
\begin{equation}
|a_k| \leq C_1 (1+ \|k\|)^{-\alpha}, \quad  k\in \mathbb{Z}^d,
\label{eq:decay-coeff-alg}
\end{equation}
\begin{equation}
|\chi ( x )| \leq C_2 (1+ \|x\|)^{-\alpha}, \quad  x\in \mathbb{R}^d.
\label{eq:decay-Lag-alg}
\end{equation}
\end{theorem}
%
\begin{proof}
Condition (\ref{eq:decay-alg}) ensures $\sigma\in J_{\alpha}$, therefore (\ref{eq:pos-symbol}) 
and Lemma~\ref{le:wiener}(i) imply that the sequence $a=\{a_k\}_{k\in\mathbb{Z}^d}$ of 
Fourier coefficients of $1/\sigma_\phi$ satisfies (\ref{eq:decay-coeff-alg}).

The decay (\ref{eq:decay-Lag-alg}) of the Lagrange function follows from the more general result of 
Bacchelli et al.\ \cite[Lemma~3]{bbrv05} on semi-discrete convolution estimates; nevertheless, we 
include the short proof, for completeness. For $x\in\mathbb{R}^d$, a direct estimate of 
(\ref{eq:CI-lagrange}), using (\ref{eq:decay-alg}) and (\ref{eq:decay-coeff-alg}), implies 
\begin{eqnarray}
|\chi (x)| 
& \leq & \sum_{k\in \mathbb{Z}^d} |a_k|\, |\phi (x-k)|  \nonumber \\
& \leq & C_1 C_0 \sum_{k\in \mathbb{Z}^d} (1+ \|k\|)^{-\alpha}(1+ \|x-k\|)^{-\alpha} .
\label{eq:CI-RHS-poly}
\end{eqnarray}
As in the proof of Lemma~\ref{le:DCP}(i), the last series is split in two parts, the first one 
corresponding to all $k$ such that $\|k\| \leq \|x\| / 2$. For such indices, we have 
$\|x-k\| \geq \|x\| / 2$, hence, for some constant $C(\alpha,d)$,
\begin{eqnarray*}
\lefteqn{ \sum_{\|k\| \leq \|x\| / 2}  (1+ \|k\|)^{-\alpha} (1+ \|x-k\|)^{-\alpha} } \\
& \leq & \sum_{\|k\| \leq \|x\| / 2}  (1+ \|k\|)^{-\alpha}(1+ \|x\|/2)^{-\alpha}  
\ \leq\  C(\alpha,d)  (1+ \|x\|)^{-\alpha} .
\end{eqnarray*}
The second part of the sum (\ref{eq:CI-RHS-poly}), for indices $k$ with $\|k\| > \|x\| / 2$, 
is bounded above by
\begin{equation*}
(1+ \|x\| / 2)^{-\alpha} \sum_{\|k\| > \|x\| / 2}  (1+ \|x-k\|)^{-\alpha}.
\end{equation*}
In turn, the last sum admits the upper bound 
\begin{equation*}
\sup_{x\in\mathbb{R}^d} \sum_{k\in \mathbb{Z}^d}  (1+ \|x-k\|)^{-\alpha} \ 
=\  \max_{x\in [0,1]^d} \sum_{k\in \mathbb{Z}^d}  (1+ \|x-k\|)^{-\alpha},
\end{equation*}
which is a finite constant depending only on $\alpha$ and $d$. 
This proves (\ref{eq:decay-Lag-alg}).
\end{proof}
\medskip

\noindent
\textbf{Remarks.} 
(i) Theorem~\ref{th:CI-decay-poly} can be adapted to the case of anisotropic decay with 
separate algebraic rates in each coordinate, using the inverse-closedness result proved 
by Gr\"{o}chenig and Klotz \cite[Theorem~1.1]{gk10} for the corresponding subalgebra of matrices 
with anisotropic off-diagonal decay.

(ii) The transfer of the weighted $\ell_1$-condition (\ref{eq:weighted-decay})  
to the associated Lagrange function $\chi$, which implies the finiteness of a weighted Lebesgue 
constant, was proved by Fageot et al.\ \cite[\S7.1]{fuw19}  
via a weighted version of Wiener's lemma (see Gr\"{o}chenig \cite[Theorem~5.24]{gro10}). 
\medskip

\noindent
\textbf{Exponential decay transfer.} 
Next, we consider the case of an exponentially decaying $\phi$, i.e.\  
we assume there exists a constant $C_0>0$ and a rate $\alpha>0$, such that 
\begin{equation}
|\phi ( x )| \leq C_0\, e^{-\alpha |x|}, \quad  x\in \mathbb{R}^d.
\label{eq:decay-exp}
\end{equation}
The use of the 1-norm of $x$ in this condition is simply an adaptation to the 
multiplicative properties of the exponential. 

%
\begin{theorem}   \label{th:CI-decay-exp}
If the continuous symmetric kernel $\phi$ satisfies \emph{(\ref{eq:decay-exp})} 
and the associated cardinal symbol $\sigma_\phi$ satisfies \emph{(\ref{eq:pos-symbol})}, 
then there exist constants $\beta\in(0,\alpha)$, $C_1=C_1(\phi,\beta,d)$, 
$C_2=C_2(\phi,\beta,\alpha,d)$, such that:
\begin{equation}
|a_k| \leq C_1\, e^{-\beta |k|}, \quad  k\in \mathbb{Z}^d,
\label{eq:decay-coeff-exp}
\end{equation}
\begin{equation}
\vert \chi( x ) \vert \leq C_2\, e^{-\beta |x|}, \quad  x\in \mathbb{R}^d.
\label{eq:decay-Lag-exp}
\end{equation}
In this case, the cardinal symbol $\sigma_\phi$ can be extended to a function of $d$ 
complex variables, which is analytic and sectorial (see the proof of 
Lemma~\ref{le:wiener}) in a neigh\-borhood of $T^d$.
\end{theorem}
%
\begin{proof}
Since $\{ \phi(k) \}_{k\in\mathbb{Z}^d}$ is the sequence of Fourier coefficients of $\sigma_\phi$, 
we have $\sigma_\phi\in\mathcal{E}_\alpha$. Hence, by the proof of Lemma~\ref{le:wiener}(ii), 
the Fourier coefficients $\{ a_k \}_{k\in\mathbb{Z}^d}$ of $1/\sigma_\phi$ satisfy 
(\ref{eq:decay-coeff-exp}), for $\beta := \alpha/2$, say.

To establish the exponential decay (\ref{eq:decay-Lag-exp}) of the Lagrange function, we use 
(\ref{eq:CI-lagrange}), (\ref{eq:decay-exp}), and (\ref{eq:decay-coeff-exp}), which imply:
\begin{eqnarray*}
|\chi (x)| & \leq & \sum_{k\in \mathbb{Z}^d} |a_k|\, |\phi (x-k)|  \\
& \leq & C_1 C_0 \sum_{k\in \mathbb{Z}^d}  \prod_{p=1}^d e^{-\beta |k_p|} e^{-\alpha |x_p - k_p|} \\
& = & C_1 C_0 \prod_{p=1}^d  \sum_{k_p\in \mathbb{Z}}  e^{-\beta |k_p|} e^{-\alpha |x_p - k_p|} ,
\quad  x\in\mathbb{R}^d.
\end{eqnarray*}
The sum over $k_p$ is estimated as in the proof of Lemma~\ref{le:DCP}(ii), using 
the triangle inequality $-|k_p|\leq |x_p-k_p|-|x_p|$:
\begin{equation*}
\sum_{k_p\in \mathbb{Z}} e^{-\beta |k_p|} e^{-\alpha |x_p - k_p|} 
\leq e^{-\beta |x_p|} \sum_{k_p\in \mathbb{Z}}  e^{-(\alpha-\beta) |x_p - k_p|} ,
\end{equation*}
and noting that the last sum admits the upper bound
\begin{equation*}
\max_{x_p\in[0,1]}  \left( e^{(\beta-\alpha) x_p} + e^{(\beta-\alpha) (1-x_p)} \right)
\sum_{\nu=0}^\infty  e^{(\beta-\alpha) \nu} = \frac{1+e^{\beta-\alpha}}{1-e^{\beta-\alpha}} =: K.
\end{equation*}
The conclusion follows, with $C_2 := C_1 C_0 K^d$.
\end{proof}
\bigskip

\noindent
\textbf{Remark.} 
Although (\ref{eq:decay-coeff-exp}) can be stated for any $\beta\in(0,\alpha)$, 
the constant $C_1$ may be arbitrarily large,  in general, as $\beta$ approaches $\alpha$.
%

\section{Interpolation on a half-space lattice}
\addtocounter{equation}{-18}

\textbf{Half-space lattices.}
In the sequel, the expression `half-space lattice' will be used to denote a subset $H$ of 
$\mathbb{Z}^d$ belonging to one of two classes. 

The first class is made of the following $d$ lattice sets:
\begin{equation*}
\{ (j_1,\ldots,j_d)\in\mathbb{Z}^d : j_p \geq 0 \},\quad p\in\{1,\ldots,d\}.
\end{equation*}
Since the results obtained are formulated similarly for any of these $d$ lattices,  
we will just consider $H = \mathbb{Z}^{d-1}\!\!\times\!\mathbb{Z}_+$ 
(for $p=d$) as a generic representative of this class of half-space lattices.

To describe a member of the second class, we assume that $\mathbb{Z}^d$ is endowed 
with a linear (or total) order relation $\preceq$ compatible with addition, in the sense that 
$j+l\preceq k+l$ whenever $j,k,l \in \mathbb{Z}^d$ and $j\preceq k$. We then let
\begin{equation*}
H = \mathbb{Z}^{d}_{\preceq,+} := \{ j\in\mathbb{Z}^d \colon 0\preceq j \}, 
\end{equation*}
the set of $\preceq$-nonnegative $d$-dimensional multi-integers. In this case, $H$ 
satisfies the following properties:
\medskip

\noindent
(a) $H+H\subset H$ (\emph{i.e.}\ $H$ is a semigroup);

\noindent
(b) $H\cap(-H) = \{0\}$;

\noindent
(c) $H\cup(-H) = \mathbb{Z}^d$.
\medskip

\noindent
Conversely, any set $H\subset\mathbb{Z}^d$ with these three properties induces 
a linear order $\preceq$ on $\mathbb{Z}^d$ compatible with addition, by simply defining 
$j\preceq k$ to mean that $k-j\in H$ (see Rudin \cite[Chapter~8]{rudinFAG}).

Note that, if $d=1$, both types of half-space lattices reduce to $H = \mathbb{Z}_+$, 
while, if $d>1$, any lattice from the first class satisfies (a) and (c), but not (b). 
When $d>1$, the second class of lattices is infinite. The following two examples 
of order relations on $\mathbb{Z}^2$ have been used by 
Goodman et al.\ \cite{gmrs00} in the context of Gram-Schmidt  
orthonormalization of the shifts of an integrable kernel:
\begin{description}
\item[(i)] 
The usual lexicographical order on $\mathbb{Z}^2$: $j\preceq k$ if either $j_1<k_1$, 
or $j_1 = k_1$ and $j_2\leq k_2$;

\item[(ii)]
A different lexicographical order on $\mathbb{Z}^2$, for which $j\preceq k$ if either 
$j_1+j_2<k_1+k_2$, or $j_1+j_2 = k_1+k_2$ and $j_2\leq k_2$.
\end{description}

\noindent
\textbf{Semi-cardinal interpolation.}
Let $H\subset\mathbb{Z}^d$ be a half-space lattice as above and suppose that the 
kernel $\phi$ satisfies the hypotheses of Theorem~\ref{th:E&U-CI}. As described in the Introduction, 
this section studies \emph{semi-cardinal interpolation} on $H$ with the kernel $\phi$, 
which is the problem of finding a bounded sequence of coefficients $c=\{c_k\}_{k\in H}$, such that 
the function $s$ of the form (\ref{eq:SCI-repres}) satisfies the interpolation conditions
\begin{equation}
s ( j ) =y_{j},\quad  j\in H,
\label{eq:SCI-problem}
\end{equation}
for a given bounded sequence of real values $y=\{y_j\}_{j\in H}$. 
The solution, constructed in \S3.2, will be expressed by means of the kernel expansion coefficients 
of Lagrange functions $\chi_j$, $j\in H$, defined by 
\begin{equation}
\chi_{j} ( x ) = \sum_{k\in H} a_{k,j}\, \phi (x-k),
\quad x\in \mathbb{R}^d,  
\label{eq:SCI-lagrange}
\end{equation}
which satisfy the interpolation conditions
\begin{equation}
\chi_{j} ( l ) =\delta _{j,l} = \left\{ 
\begin{array}{rr}
1, & l=j, \\ 
0, & l\in H \setminus \{j\}.
\end{array}
\right.
\label{eq:SCI-delta}
\end{equation}
This is equivalent to finding, for each $j\in H$, the sequence $\{a_{k,j}\}_{k\in H}$ 
satisfying the Wiener-Hopf (or semi-infinite Toeplitz) system of difference equations
\begin{equation}
\sum_{k\in H} a_{k,j} \phi (l-k) = \delta_{j,l}, \quad l\in H.
\label{eq:SCI-system-Lag}
\end{equation}

In \S3.1, the explicit solution of this system is obtained using the 
Wiener-Hopf factorization of the cardinal symbol $\sigma_\phi$ (equivalently, of its reciprocal) 
generated by $\phi$. This method is due, for $d=1$, to Krein \cite{mgk58} and
Calderón et al.\ \cite{csw59}. For $d>1$, the case 
$H = \mathbb{Z}^{d-1}\!\!\times\!\mathbb{Z}_+$ was treated by Goldenstein 
and Gohberg \cite{gg60} and Goldenstein \cite{lsg64}. In the case 
$H = \mathbb{Z}^{d}_{\preceq,+}$, the Wiener-Hopf factorization technique was studied by 
Goodman et al.\ \cite{gmrs00} for obtaining Cholesky factorizations of bi-infinite 
Gram-Laurent matrices generated by shifts of multivariate kernels, while the application of 
this technique to the solution of more general multi-index semi-infinite block Toeplitz systems  
was considered in van der Mee et al.\ \cite[p.\ 467]{vsr03}.

The transfer of specific algebraic or exponential decay of the kernel $\phi$ to the 
family of Lagrange functions $\chi_j$, $j\in H$, is obtained in \S3.3.
\medskip

\noindent
\textbf{Notation remarks.}
For simplicity, in this section we remove the dependence on $\phi$ in most notation, e.g.\ we will 
use $\sigma:=\sigma_\phi$; nevertheless, this dependence remains implicit throughout the section. 
Also, for a $d$-dimensional vector $x$, we employ the partition notation $x=(x',x_d)$.

\subsection{The Wiener-Hopf factorization}

The construction of the Wiener-Hopf (also known as `spectral') factorization of 
$\omega=1/\sigma\in W$ is based on the fact that, under the hypotheses of 
Theorem~\ref{th:E&U-CI}, $\log\omega$ also belongs to the Wiener algebra. Indeed, this is 
ensured by the following multivariable version of L\'{e}vy's extension of Wiener's lemma, 
which is a special case (for the trivial weight) of a result proved by Goodman 
et.\ al.\ \cite[Theorem~2.3]{gmrs00} for certain weighted Wiener algebras.
%
\begin{lemma}[Wiener-L\'{e}vy]
\label{le:wiener-levy1}
If $\psi\in W$ and $F$ is analytic in a neighborhood of the range of $\psi$, then $F\circ\psi\in W$. 
\end{lemma}
%
Since $\omega=1/\sigma$ is positive on $T^d$ due to (\ref{eq:pos-symbol}), letting $\psi:=\omega$ 
and $F:=\log$, it follows that $\log\omega\in W$, with an absolutely convergent Fourier expansion:
\begin{equation}
\log \omega ( \zeta ) =\sum_{k\in \mathbb{Z}^d} \lambda _{k} \zeta^{k},\quad \zeta\in T^d.
\label{eq:log-omega}
\end{equation}
Note that $\lambda_{-k}=\lambda _{k}\in\mathbb{R}$, $\forall k\in \mathbb{Z}^d$, since the Fourier 
coefficients of $\omega$ are also real-valued and symmetric. The required Wiener-Hopf factorization 
of $\omega$ is now a consequence of the decomposition 
\begin{equation*}
\log \omega ( \zeta ) = \Lambda_+(\zeta) + \Lambda_+(\zeta^{-1}),
\end{equation*}
which is obtained by selecting the function $\Lambda_+$ in the following way:

If $H = \mathbb{Z}^{d-1}\!\!\times\!\mathbb{Z}_+$, we set, as in Goldenstein and 
Gohberg \cite{gg60},
\begin{equation}
\Lambda _{+} (\zeta)  :=  \frac{1}{2} \sum_{k'\in\mathbb{Z}^{d-1}} \lambda _{(k',0)} (\zeta')^{k'} 
+ \sum_{k_d=1}^\infty \sum_{k'\in\mathbb{Z}^{d-1}} \lambda _{k} \zeta^{k} ,\quad \zeta\in T^d.
\label{eq:lambda+} 
\end{equation}

If $H = \mathbb{Z}^{d}_{\preceq,+}$, we define, as in Goodman et al.\ \cite{gmrs00},
\begin{equation}
\Lambda _{+} (\zeta)  :=  \frac{1}{2} \lambda _0 
+ \sum_{k\succ 0} \lambda _{k} \zeta^{k} ,\quad \zeta\in T^d.
\label{eq:lambda++} 
\end{equation}
In both cases, letting
\begin{equation}
\omega _{+} (\zeta)  :=  \exp [ \Lambda _{+} (\zeta) ] ,
\label{eq:omega+}
\end{equation}
leads to the following Wiener-Hopf (or canonical) factorization:
\begin{equation}
\omega (\zeta) = e^{\log \omega ( \zeta )} = e^{\Lambda_+(\zeta) + \Lambda_+(\zeta^{-1})} 
= \omega_+ (\zeta) \omega_+ (\zeta^{-1}), \quad \zeta\in T^d.
\label{eq:wh}
\end{equation}
Invoking again the Wiener-L\'{e}vy lemma, this time for the composite function (\ref{eq:omega+}), 
it follows that $\omega_+\in W$. Moreover, using the power series expansion of the exponential 
and the fact that $H$ is closed under addition, we deduce that the Fourier coefficients of 
$\omega_+$, as those of $\Lambda_+$, are supported on $H$:
\begin{equation}
\omega _{+} (\zeta) = \sum_{k\in H} \gamma_{k} \zeta^{k},
\quad \zeta\in T^d.
\label{eq:om+Fourier}
\end{equation}
The same arguments show that $\omega_+^{-1}=\exp(-\Lambda_+)\in W$ and the Fourier coefficients 
of $\omega_+^{-1}$ are supported on $H$, as well.
\medskip

\noindent
\textbf{Remark.} The existence of a multi-index Wiener-Hopf factorization of a non-vanishing symbol 
$\omega$ is obtained in the literature under more general conditions than positivity, expressed in terms 
of the winding number $\textup{wind}_p (\omega)$ of $\omega$ with respect to $\zeta_p$ about the origin. 
Specifically, if $H = \mathbb{Z}^{d-1}\!\!\times\!\mathbb{Z}_+$, then $\textup{wind}_d (\omega) = 0$ is 
required (Goldenstein and Gohberg \cite{gg60}), while, if $H = \mathbb{Z}^{d}_{\preceq,+}$, the extra 
condition is $\textup{wind}_p (\omega) = 0$, for all $p = 1,\ldots,d$, which is seen to be independent 
of the underlying linear order (Ehrhardt and van der Mee \cite{ev03}).
\medskip

\noindent
\textbf{Semi-cardinal symbols.}
Let $P_+$ denote the truncation operator defined by
\begin{equation*}
P_+ \left\{ \sum_{k\in\mathbb{Z}^{d}} \widehat{\psi}_{k} \zeta^{k} \right\}
= \sum_{k\in H} \widehat{\psi}_{k} \zeta^{k},
\quad \zeta\in T^d,
\end{equation*}
for any $\psi\in W$ with Fourier coefficients $\{\hat{\psi}_k\}_{k\in\mathbb{Z}^d}$.

With the factorization (\ref{eq:wh}) in hand, we now follow Krein's method  \cite{mgk58} 
(for $d=1$) to express, for each $j\in H$, the solution $\{a_{k,j}\}_{k\in H}$ of the Wiener-Hopf system 
(\ref{eq:SCI-system}) as the sequence of Fourier coefficients of the symbol:
\begin{equation}
\omega _{j} (\zeta) := \omega _{+} (\zeta) P_{+}\left\{ \zeta^{j} \omega _{+} (\zeta^{-1}) \right\}, 
\quad \zeta\in T^d.
\label{eq:sc-symbol}  
\end{equation}
The second factor can be seen as a twisted truncation of the symbol $\omega _{+} (\zeta^{-1})$, 
which appears in the Wiener-Hopf factorization (\ref{eq:wh}). 
Indeed, if $H = \mathbb{Z}^{d}_{\preceq,+}$, then for each $j\succeq 0$, we have
\begin{equation}
P_+ \{ \zeta^j \omega_+ (\zeta^{-1}) \} = \zeta^j \omega_{-}^{(j)} (\zeta),
\label{eq:twisted-trunc}
\end{equation}
where $\omega_{-}^{(j)}$ is the truncation of the symbol $\omega_+ (\zeta^{-1})$ defined by
\begin{equation}
\omega_{-}^{(j)} (\zeta) := \sum_{0\preceq l\preceq j} \gamma_l \zeta^{-l}.
\label{eq:om+truncation}
\end{equation}
Note that the last sum can have an infinite number of terms when $d>1$. More importantly, 
the following uniform bound holds:
\begin{equation} 
\| \omega_{-}^{(j)} \|_W \leq \| \omega_{+} \|_W, \quad \forall j\in H.
\label{eq:om+tr-bound}
\end{equation}

Also, if $H = \mathbb{Z}^{d-1}\!\!\times\!\mathbb{Z}_+$, replacing (\ref{eq:om+truncation}) 
by the truncation with respect to the last coordinate:
\begin{equation}
\omega_{-}^{(j_d)} (\zeta) := \sum_{l_d = 0} ^{j_d} \sum_{l'\in \mathbb{Z}^{d-1}} \gamma_l \zeta^{-l},
\label{eq:om+trunc}
\end{equation}
it follows that (\ref{eq:twisted-trunc}) and (\ref{eq:om+tr-bound}) hold with $\omega_{-}^{(j_d)}$ 
in place of $\omega_{-}^{(j)}$.

Clearly, for both choices of $H$, we have $\omega_j \in W$ and the Fourier coefficients of 
$\omega_j$ are supported on $H$, since both factors of (\ref{eq:sc-symbol}) possess this 
property. Hence, $\omega_j$ admits an absolutely convergent Fourier representation:
\begin{equation}
\omega _{j} (\zeta) = \sum_{k\in H} a_{k,j}\, \zeta^{k}, \quad \zeta\in T^d.
\label{eq:sc-symb-F}
\end{equation}
However, it may not be immediately obvious that $\{a_{k,j}\}_{k\in H}$ is a solution of 
(\ref{eq:SCI-system-Lag}). This claim is settled next.
%
\begin{proposition}    \label{prop:sc-Lagrange}
For each $j\in H$, the sequence $\{a_{k,j}\}_{k\in H}$ of Fourier coefficients of $\omega_j$ 
satisfies the semi-infinite system \textup{(\ref{eq:SCI-system-Lag})}. 
\end{proposition}
%
\begin{proof}
%
Let $\overline{H}:=\mathbb{Z}^{d}\setminus H$. Using the absolutely convergent Fourier 
expansion (\ref{eq:om+Fourier}) of $\omega_+$ and the Wiener-Hopf factorization (\ref{eq:wh}), 
definition (\ref{eq:sc-symbol}) implies:
\begin{eqnarray*}
\omega_j (\zeta) & = & \omega_+ (\zeta) \,
P_{+}\left\{ \zeta^{j} \sum_{k\in H} \gamma_{k} \zeta^{-k} \right\}  \\
& = & \omega_+ (\zeta) \left\{ \zeta^{j} \omega _{+} (\zeta^{-1}) 
- \sum_{l\in\overline{H}} \gamma_{j-l} \zeta^{l} \right\} \\
& = & \zeta^{j} \omega (\zeta) - \omega_+ (\zeta) \sum_{l\in\overline{H}} \gamma_{j-l} \zeta^{l} ,
\end{eqnarray*}
or, equivalently,
\begin{equation*}
\omega_j (\zeta) \sigma (\zeta) = \zeta^{j} - 
 \omega_+^{-1} (\zeta^{-1}) \sum_{l\in\overline{H}} \gamma_{j-l} \zeta^{l} .
\end{equation*}
Since $j\in H$ and the second term of the right-hand side admits an absolutely convergent 
expansion indexed over $\overline{H}$, it follows that 
\begin{equation*}
P_+ \left\{ \omega_j (\zeta) \sigma (\zeta) \right\} = \zeta^{j}.
\end{equation*}

On the other hand, using the Fourier expansions of $\omega_j$ and $\sigma$, we have
\begin{eqnarray*}
P_+ \left\{ \omega_j (\zeta) \sigma (\zeta) \right\} & = &
\sum_{k\in H} a_{k,j}\, \zeta^{k} \sum_{l\in H} \phi (l-k)\, \zeta^{l-k} \\ 
& = & \sum_{l\in H} \sum_{k\in H} a_{k,j}\, \phi (l-k)\, \zeta^{l}.
\end{eqnarray*}
The system (\ref{eq:SCI-system-Lag}) now follows by comparing the right-hand sides of the 
last two displays, due to the orthogonality of the trigonometric system.
\end{proof}
\medskip

Next, we obtain an explicit formula for the Fourier coefficients of the semi-cardinal symbols 
$\omega_j$, $j\in H$, in terms of those of the Wiener-Hopf factor $\omega_+$; the symmetry 
of these coefficients follows as a consequence. We also establish a crucial Schur-type property, 
equivalent to the uniform boundedness of the Wiener norms $\|\omega_j\|_W$, for $j\in H$.
%
\begin{proposition}		\label{le:sc-symbol}
Assume that $\phi$ satisfies the hypotheses of \emph{Theorem~\ref{th:E&U-CI}}.

\textup{(i)} If $H = \mathbb{Z}^{d}_{\preceq,+}$, then
\begin{equation}
a_{k,j} = \sum_{ 0\preceq l\preceq \min \{ j,k \} } \gamma_{k-l} \gamma_{j-l}, 
\quad j,k\in \mathbb{Z}^{d}_{\preceq,+}.
\label{eq:sc-coeffs+}
\end{equation}
If $H = \mathbb{Z}^{d-1}\!\!\times\!\mathbb{Z}_{+}$, then, for all $j=(j',j_d)$ and $k=(k',k_d)$ in $H$, 
\begin{eqnarray}
\omega_{(j',j_d)} (\zeta) & = & (\zeta')^{j'} \omega_{(0,j_d)} (\zeta), \quad \zeta\in T^d,
\label{eq:sc-symm} \\
a_{k,j} & = & a_{(k'-j',k_d),(0,j_d)}, 
\label{eq:invariance-c} \\
a_{k,j} & = & \sum_{l_d=0}^{\min\{j_d,k_d\}} 
\sum_{l'\in\mathbb{Z}^{d-1}} \gamma_{(k'-l',k_d-l_d)} \gamma_{(j'-l',j_d-l_d)}.
\label{eq:sc-coeffs}
\end{eqnarray}

\textup{(ii)} For both types of half-space lattices $H$, we have $a_{k,j} = a_{j,k}$, for all $j,k \in H$, 
and the following Schur property holds:
\begin{equation}
\sup_{j\in H} \sum_{k\in H} |a_{k,j}| < \infty .
\label{eq:sc-schur}
\end{equation}
\end{proposition}
%
\begin{proof}
(i) Let $H = \mathbb{Z}^{d}_{\preceq,+}$. Using (\ref{eq:om+Fourier}) and (\ref{eq:sc-symbol}), 
we have, for each $j\in \mathbb{Z}^{d}_{\preceq,+}$,
\begin{eqnarray*}
\omega_j (\zeta) 
& = & \sum_{k\succeq 0} \gamma_k \zeta^k  \sum_{0\preceq l\preceq j} \gamma_{j-l} \zeta^l  \\
& = & \sum_{k\succeq 0} \zeta^k  \sum_{ 0\preceq l\preceq \min \{ j,k \} } \gamma_{k-l} \gamma_{j-l},
\end{eqnarray*}
which implies (\ref{eq:sc-coeffs+}). 

Next, let $H = \mathbb{Z}^{d-1}\!\!\times\!\mathbb{Z}_{+}$. Relation (\ref{eq:sc-symm}) 
follows from (\ref{eq:sc-symbol}), since, for any $p\in\{1,\ldots,d-1\}$, multiplication by 
$\zeta_p$ commutes with the truncation operator $P_+$. Also, (\ref{eq:sc-symm}) implies 
(\ref{eq:invariance-c}). To obtain (\ref{eq:sc-coeffs}), we use re-arrangements of absolutely 
convergent $d$-dimensional Fourier series with respect to the last component of the 
summation multi-index. Specifically, we start with the Fourier expansion (\ref{eq:om+Fourier}) 
of $\omega_+$, written as
\begin{equation*}
\omega _{+} (\zeta) = \sum_{n=0}^{\infty} \Gamma_{n} (\zeta')\, \zeta_d^n,
\end{equation*}
where $\Gamma_n$ is the $(d-1)$-dimensional absolutely convergent Fourier series
\begin{equation*}
\Gamma _{n} (\zeta') :=\sum_{k' \in\mathbb{Z}^{d-1}} \gamma_{(k',n)} (\zeta')^{k'},
\quad n\in\mathbb{Z}_+.
\end{equation*}
Then, for $j_d\geq 0$, the truncation $\omega_{-}^{(j_d)}$ of $\omega_+ (\zeta^{-1})$ defined 
in (\ref{eq:om+trunc}) acquires the form
\begin{equation*}
\omega_{-}^{(j_d)} (\zeta) = \sum_{n = 0} ^{j_d}  \Gamma_n \! \left( (\zeta')^{-1} \right) \zeta_d^{-n},
\end{equation*}
which implies 
\begin{equation*}
P_{+}\left\{ \zeta_d^{j_d} \omega _{+} (\zeta^{-1}) \right\} 
\ =\  \zeta_d^{j_d} \omega_{-}^{(j_d)} (\zeta)  
\ =\  \sum_{n = 0} ^{j_d}  \Gamma_{j_d-n} \! \left( (\zeta')^{-1} \right) \zeta_d^{n}.
\end{equation*}
Thus, using (\ref{eq:sc-symbol}) for $j'=0$, we obtain
\begin{eqnarray*}
\omega_{(0,j_d)} (\zeta) 
& = & \omega _{+} (\zeta) P_{+}\left\{ \zeta_d^{j_d} \omega _{+} (\zeta^{-1}) \right\}  \\
& = & \sum_{k_d=0}^{\infty} \Gamma_{k_d} (\zeta') \zeta_d^{k_d}  
\sum_{n = 0} ^{j_d}  \Gamma_{j_d-n} \! \left( (\zeta')^{-1} \right) \zeta_d^{n}  \\
& = & \sum_{k_d=0}^{\infty} A_{k_d,j_d} (\zeta') \zeta_d^{k_d} ,
\end{eqnarray*}
where
\begin{equation*}
A_{k_d,j_d} (\zeta') := \sum_{n = 0} ^{\min\{k_d,j_d\}}  
\Gamma _{k_d-n} (\zeta') \Gamma_{j_d-n} \! \left( (\zeta')^{-1} \right).
\end{equation*}
It follows that the Fourier coefficient of $\omega_{(0,j_d)}$ of index $k=(k',k_d)$ is given explicitly by
\begin{equation*}
a_{(k',k_d),(0,j_d)} = \sum_{l_d=0}^{\min\{j_d,k_d\}} 
\sum_{l'\in\mathbb{Z}^{d-1}} \gamma_{(k'-l',k_d-l_d)} \gamma_{(-l',j_d-l_d)}.
\end{equation*}
Now, (\ref{eq:sc-coeffs}) is a consequence of  the above formula and (\ref{eq:invariance-c}).

(ii) The symmetry of the coefficients follows directly from the explicit formulae 
(\ref{eq:sc-coeffs+}) and (\ref{eq:sc-coeffs}). Note that (\ref{eq:sc-schur}) is equivalent to 
$\sup_{j\in H} \| \omega_j \|_W < \infty$. To prove this, we use  the expression 
(\ref{eq:sc-symbol}) of $\omega_j$ as a product, together with the Wiener algebra property 
$\|uv\|_W \leq \|u\|_W \|v\|_W$, valid for any $u,v \in W$. Specifically, if 
$H = \mathbb{Z}^{d}_{\preceq,+}$, then, due to (\ref{eq:twisted-trunc}) and the uniform 
bound (\ref{eq:om+tr-bound}) on the Wiener norm of the twisted truncation factor of 
(\ref{eq:sc-symbol}), we have
\begin{equation*}
\|\omega_j\|_W \leq \|\omega_+\|_W \|\omega_{-}^{(j)} \|_W \leq \|\omega_+\|_W^2 < \infty, 
\quad j\in H.
\end{equation*}
A similar proof applies to the case $H = \mathbb{Z}^{d-1}\!\!\times\!\mathbb{Z}_{+}$, since 
(\ref{eq:twisted-trunc}) and (\ref{eq:om+tr-bound}) hold with $\omega_{-}^{(j_d)}$ in place 
of $\omega_{-}^{(j)}$.
\end{proof}

\subsection{The semi-cardinal interpolation scheme}
\noindent
\textbf{Semi-cardinal Lagrange functions.}
For each $j\in H$, the Fourier coefficients $\{a_{k,j}\}_{k\in H}$ of the symbol $\omega_j$ 
are now used to define $\chi_j$ via (\ref{eq:SCI-lagrange}). That series is absolutely and 
uniformly convergent on $\mathbb{R}^d$, due to the boundedness of $\phi$ and the absolute 
summability of $\{a_{k,j}\}_{k\in H}$. Further, Proposition~\ref{prop:sc-Lagrange} shows that 
$\chi_j$, $j\in H$, are the expected Lagrange functions for interpolation on $H$, since they satisfy the 
interpolation conditions (\ref{eq:SCI-delta}).
\medskip

\noindent
\textbf{Remark.} 
When $H = \mathbb{Z}^{d-1}\!\!\times\!\mathbb{Z}_{+}$, property (\ref{eq:invariance-c}) 
implies, via re-indexing, the translation symmetry of $\chi_j$ for shifts parallel to the boundary 
of the half-space domain $\Omega = \mathbb{R}^{d-1} \times [0,\infty)$, as follows:
\begin{equation}
\chi_{(j',j_d)} (x',x_d) = \chi_{(0,j_d)} (x' - j' , x_d),\quad \forall x\in\mathbb{R}^d,\ 
j\in \mathbb{Z}^{d-1}\times\mathbb{Z}_{+}.
\label{eq:invariance-L}
\end{equation}

We now formulate and prove the existence and uniqueness properties of the semi-cardinal 
interpolation scheme under minimal conditions on the kernel $\phi$.
%
\begin{theorem}      \label{th:E&U-SCI}
Assume that $\phi$ satisfies the hypotheses of \emph{Theorem~\ref{th:E&U-CI}}, and let the half-space 
lattice $H$ be either $\mathbb{Z}^{d-1}\!\!\times\!\mathbb{Z}_+$ or $\mathbb{Z}^{d}_{\preceq,+}$. 
Then, for any bounded sequence of data values $\{y_j\}_{j\in H}$, there exists a unique bounded 
sequence of coefficients $\{c_j\}_{j\in H}$, such that the continuous function \textup{(\ref{eq:SCI-repres})}
satisfies the semi-cardinal interpolation system \textup{(\ref{eq:SCI-problem})}. Specifically, we have:
\begin{equation}
c_k = \sum_{j\in H} a_{k,j} y_j ,\quad k\in H,
\label{eq:SCI-sol-coeffs}
\end{equation}
where $a_{k,j}$ is the corresponding coefficient in the representation \emph{(\ref{eq:SCI-lagrange})} 
of $\chi_j$. The unique semi-cardinal interpolant $s$ admits the Lagrange representation 
\begin{equation}
s ( x ) =\sum_{j\in H} y_{j} \chi_{j} ( x ) ,
\quad x\in\mathbb{R}^d.
\label{eq:SCI-scheme}
\end{equation}
Moreover, if $I_{\phi}^H : \ell_{\infty} (H)\rightarrow L_{\infty} (\mathbb{R}^d)$ denotes the operator 
that associates to each $y\in\ell_{\infty} (H)$ its unique semi-cardinal interpolant $I_{\phi}^H y:=s$, 
as above, then $I_{\phi}^H$ is a linear bounded operator and its norm (Lebesgue constant) satisfies
\begin{equation*}
\| I_{\phi}^H \|_{\infty} = \sup_{x\in\mathbb{R}^d}  \sum_{j\in H} |\chi_j ( x )| 
\leq  |\phi|_\infty \sup_{j\in H} \|\omega_j\|_W < \infty.
\end{equation*}
\end{theorem}
%

%
\begin{proof}
We argue along the lines of the proof corresponding to the cardinal interpolation case, the crucial 
difference being that the finiteness of the Wiener norm $\|\omega\|_W$ is replaced here by the 
uniform boundedness of $\|\omega_j\|_W$, for $j\in H$. First, consider the homogeneous system 
corresponding to the SCI problem (\ref{eq:SCI-problem}):
\begin{equation*}
\sum_{k\in H} c_k \phi (j-k) = 0,
\quad j\in H,
\end{equation*}
for a bounded sequence $\{c_k\}_{k\in H}$. For every $l\in H$, multiplying each of the above 
equations by $a_{j,l}$ and summing over $j$, we obtain
\begin{eqnarray*}
0 & = & \sum_{j\in  H} a_{j,l} \sum_{k\in  H} c_k \phi (j-k)  
 \ = \ \sum_{k\in  H} c_k  \sum_{j\in  H} a_{j,l} \phi (k-j)    \\
& = & \sum_{k\in  H} c_k \chi_l (k)
 \ = \ \sum_{k\in  H} c_k \delta_{l,k} \ = \ c_l,
\end{eqnarray*}
where we have used the fact that $\phi$ is symmetric, as well as (\ref{eq:SCI-delta}). 
The exchange of summation in the second equality above is permitted by virtue of 
the absolute convergence of the double sum:
\begin{eqnarray*}
\sum_{j,k\in  H} |a_{j,l} c_k \phi (j-k)|  
& = & \sum_{j\in  H} | a_{j,l} |  \sum_{k\in  H} |c_k| |\phi (j-k)|   \\
& \leq & \|c\|_{\infty}  \sum_{m\in \mathbb{Z}^d} |\phi (m)|  \,
\sup_{k\in H}  \sum_{j\in  H} |a_{j,k}|  \ < \ \infty, 
\end{eqnarray*}
based on the Schur property (\ref{eq:sc-schur}) and on (\ref{eq:minimal}). This establishes the 
uniqueness.

To prove existence, note that the right-hand side of (\ref{eq:SCI-scheme}), which evidently 
satisfies the cardinal interpolation conditions (\ref{eq:SCI-problem}), can be written in the 
form (\ref{eq:SCI-repres}), for all $x\in\mathbb{R}^d$:
\begin{eqnarray*}
\sum_{j\in  H} y_{j} \chi_j (x) 
& = & \sum_{j\in  H} y_j  \sum_{k\in H} a_{k,j} \phi (x-k)   \\
& = & \sum_{k\in  H} c_k \phi (x-k),
\end{eqnarray*}
where $c_k$ is given by (\ref{eq:SCI-sol-coeffs}) and the interchange of summation is again 
permitted by the absolute convergence of the double series:
\begin{eqnarray*}
\sum_{j,k\in  H} |y_j a_{k,j} \phi (x-k)|  
& = & \sum_{k\in  H} |\phi (x-k)|  
\sum_{j\in  H} |y_j| |a_{k,j}|  \\
& \leq & \|y\|_{\infty} \left( \sup_{k\in H} \sum_{j\in  H} |a_{k,j}| \right)
\sup_{x\in\mathbb{R}^d} \sum_{k\in H} |\phi (x-k)| \\
& \leq & \|y\|_{\infty} \left( \sup_{k\in H} \|\omega_k\|_W \right) |\phi|_\infty ,
\end{eqnarray*}
due to (\ref{eq:sc-schur}), the symmetry property $a_{k,j} = a_{j,k}$, and the hypothesis 
(\ref{eq:minimal}). By uniqueness, it follows that the right-hand side of 
(\ref{eq:SCI-scheme}) is the required semi-cardinal interpolant to the data sequence.

Turning to the last statement of the theorem, note that the last paragraph implies, in particular,
\begin{equation*}
\| I_{\phi}^H \|_{\infty} \leq \sup_{x\in\mathbb{R}^d} \Lambda_H (x) 
\leq |\phi|_\infty  \sup_{k\in H} \|\omega_k\|_W, 
\end{equation*}
where $\Lambda_H (x):=\sum_{j\in H} | \chi_j ( x ) |$ is the continuous `Lebesgue function' 
for semi-cardinal interpolation on $H$. To show that on the left we have in fact equality, 
let $\{x_n\}_{n=1}^\infty$ be a sequence of points 
in $\mathbb{R}^d$, such that $\Lambda_H (x_n) \to \sup_{x\in\mathbb{R}^d} \Lambda_H (x)$, 
as $n\to\infty$. For each $n$, choose a sequence $y^{(n)}=\{y^{(n)}_j\}_{j\in H}$ with terms 
$y^{(n)}_j = \pm 1$, $\forall j\in H$, such that $y^{(n)}_j \chi_j (x_n) = | \chi_j (x_n) |$, 
$\forall j\in H$. Since $\| y^{(n)} \|_\infty = 1$, we have
\begin{equation*}
\Lambda_H (x_n) = \sum_{j\in H} y^{(n)}_j \chi_j (x_n) 
= I_{\phi}^H y^{(n)} (x_n)  \leq  \| I_{\phi}^H y^{(n)} \|_{\infty}  \leq \| I_{\phi}^H \|_{\infty}.
\end{equation*}
The proof is completed by letting $n\to\infty$.
\end{proof}

\begin{corollary}
For a function $s$ of the form \emph{(\ref{eq:SCI-repres})}, where $c\in\ell_{\infty} ( H)$, 
the following reproduction formula holds:
\begin{equation*}
s ( x ) =\sum_{j\in H} s(j) \chi_j ( x ) ,\quad x\in\mathbb{R}^d.
\end{equation*}
In particular, this applies to the kernel shifts $s:=\phi (\cdot - k)$, with $k\in H$, hence, 
for $k=0$, we have $\phi ( x ) =\sum_{j\in H} \phi(j) \chi_j ( x )$, for all $x\in\mathbb{R}^d$.
\end{corollary}

Next, we provide the Toeplitz operator interpretation of Theorem~\ref{th:E&U-SCI}, by considering 
the convolution operator represented by the semi-infinite multi-index Toeplitz matrix 
$T_{\phi} = [\phi (j-k)]_{j,k\in H}$, i.e.\ for each bounded sequence $c=\{c_k\}_{k\in H}$, 
we set $T_{\phi} c = y=\{y_j\}_{j\in H}$, where $y_j = \sum_{k\in H} \phi(j-k) c_k$, $j\in H$.
Also, let $G = [G_{kj}]_{k,j\in H}$ be the matrix defined by $G_{kj} = \gamma_{k-j}$ for $k,j\in H$, 
where $\{\gamma_k\}_{k\in H}$ is the sequence of Fourier coefficients of the symbol 
$\omega_+$ (recall (\ref{eq:om+Fourier})), and where we set $\gamma_k = 0$ for 
$k\in\mathbb{Z}^d\setminus H$.

\begin{corollary}	\label{cor:SCI-operator}
Under the hypotheses of \textup{Theorem~\ref{th:E&U-SCI}}, the linear operator $T_{\phi}$ is 
bounded and invertible on $\ell_p (H)$, for all $p\in [1,\infty]$. Further, the inverse operator 
$T_{\phi}^{-1} = [a_{k,j}]_{k,j\in H}$ admits the factorization:
\begin{equation}
T_{\phi}^{-1} = G G^{T} ,
\label{eq:cholesky}
\end{equation}
where $G$ is also bounded and invertible on $\ell_p(H)$ and $G^T$\! denotes its transpose. 
\end{corollary}

\begin{proof}
The classical Schur test (see Gr\"{o}chenig \cite{gro01}) provides the inequality
\begin{equation*}
\| T_{\phi} c \|_p  \leq  \|  c \|_p\,  \sup_{j\in H} \sum_{k\in H} |\phi (j-k)|
\leq  \| c \|_p \sum_{k\in \mathbb{Z}^d} |\phi (k)|, \quad c\in\ell_p (H),
\end{equation*}
showing that $T_{\phi} \colon \ell_p (H) \to \ell_p (H)$ is a bounded linear operator. Further, 
Theorem~\ref{th:E&U-SCI} implies that $T_{\phi}$ is one-to-one and onto. Then the Banach 
Inverse Theorem guarantees that $T_{\phi}^{-1}$ is bounded. This fact also follows directly 
from (\ref{eq:sc-schur}) by the Schur test, since formula 
(\ref{eq:SCI-sol-coeffs}) shows that the column of index $j$ of $T_{\phi}^{-1}$ is the 
vector of Fourier coefficients of the symbol $\omega_j$. 

To obtain the factorization of $T_{\phi}^{-1}$, note that, with the above 
convention $\gamma_k = 0$ for $k\not\in H$, both explicit formulae 
(\ref{eq:sc-coeffs+}) and (\ref{eq:sc-coeffs}) take the same form:
\begin{equation*}
a_{k,j} = \sum_{l\in H} \gamma_{k-l} \gamma_{j-l}, \quad k,j\in H.
\end{equation*}
Hence, the definition of the matrix $G$ implies 
$a_{k,j} = \sum_{l\in H} G_{kl} G_{jl} = (GG^{T})_{kj}$ for all $k,j\in H$, 
which proves (\ref{eq:cholesky}). Another application of the Schur test shows that $G$ is 
also bounded on $\ell_p (H)$. Further, one can verify that the bounded inverse 
of $G$ is given by $G^{-1} = [\widetilde{\gamma}_{k-j}]_{k,j\in H}$, where 
$\{\widetilde{\gamma}_k\}_{k\in H}$ is the sequence of Fourier coefficients of the symbol 
$\omega_+^{-1}$ and $\widetilde{\gamma}_k := 0$ for $k\in\mathbb{Z}^d\setminus H$.
\end{proof}
\medskip

\noindent
\textbf{Remark.} 
For $d=1$, the factorization (\ref{eq:cholesky}) was given by Krein \cite[(13.27)]{mgk58}. 
For $d>1$, there is an important distinction between the two types of half-space lattices.
Indeed, if $H = \mathbb{Z}^{d}_{\preceq,+}$, for a linear order $\preceq$, then 
the matrix $G$ and its inverse $G^{-1}$ are lower triangular, in the sense that 
$G_{kj} = 0 = (G^{-1})_{kj}$ for $k\prec j$. In this case, (\ref{eq:cholesky}) provides the Cholesky 
factorization of the inverse matrix $T_{\phi}^{-1}$  relative to $\preceq$. 
However, if $H = \mathbb{Z}^{d-1}\!\!\times\!\mathbb{Z}_+$, it can be readily seen that 
the lower triangular structure of $G$ is lost.
\subsection{Decay transfer}

We now impose specific algebraic or exponential decay conditions on $\phi$ of the form 
considered in \S2.2, and we prove that these carry over to all semi-cardinal Lagrange functions 
$\chi_j$, $j\in H$, with constants independent of $j$. 
To this aim, we will first obtain the decay transfer to the Fourier 
coefficients of the semi-cardinal symbols $\omega_j$, $j\in H$, i.e.\ to the inverse matrix 
$T_{\phi}^{-1} = [a_{k,j}]_{k,j\in H}$. 
\medskip

\noindent
\textbf{Decay transfer to Fourier coefficients of semi-cardinal symbols}
%

\begin{theorem}      \label{th:SCI-decay-symb}
Assume that $\phi$ satisfies the hypotheses of \emph{Theorem~\ref{th:E&U-CI}}, and let 
$H$ be a half-space lattice.

\textbf{\emph{(i) Algebraic decay.}} If, in addition, $\phi$ decays algebraically according to  
\emph{(\ref{eq:decay-alg})}, with power $\alpha>d$, then both the Fourier coefficients  
of the Wiener-Hopf factor $\omega_+$ and those of the semi-cardinal symbols $\omega_j$ 
decay with the same power, i.e.\ there exists $C_{1,2}=C_{1,2} (\alpha,d,\phi)>0$, such that:
\begin{eqnarray}
| \gamma _{k} | & \leq & C_1 (1+ \|k\|)^{-\alpha}, \quad k\in H,
\label{eq:om+coeffs-poly} \\
| a_{k,j} | &\leq & C_2 (1+\|k-j\|)^{-\alpha} ,\quad  j,k\in H.
\label{eq:sc-symb-decay-poly}
\end{eqnarray}

\textbf{\emph{(ii) Exponential decay.}} If $\phi$ decays exponentially as in  
\emph{(\ref{eq:decay-exp})}, with a rate $\alpha>0$, then there exist constants 
$\beta\in(0,\alpha)$, $C_1=C_1(\beta,\phi,H,d)$, such that 
\begin{equation}
\vert \gamma _{k} \vert \leq C_1 \, e^{-\beta |k|}, \quad  k\in H.
\label{eq:om+coeffs}
\end{equation}
Also, for each $\beta'\in(0,\beta)$, there exists $C_{2}=C_{2}(\beta,\beta',\phi,H,d)$, 
such that
\begin{equation}
\vert a_{k,j} \vert \leq C_2 \,  e^{-\beta' |k-j|}, \quad  j,k\in H,
\label{eq:sc-symb-decay}
\end{equation}
where the right-hand side exponential of the last inequality can be replaced by 
$e^{-\beta' |k'-j'|-\beta |k_d-j_d|}$, if $H=\mathbb{Z}^{d-1}\times\mathbb{Z}_{+}$.
\end{theorem}
%

Below, we will provide two proofs for this theorem. The first proof is constructive, using the explicit 
form (\ref{eq:sc-symbol}) of the semi-cardinal symbols, which is based on the Wiener-Hopf 
factorization (\ref{eq:wh}). As in the cardinal interpolation case, our arguments 
for exponential decay rely on analyticity. However, for the transfer of algebraic 
decay we need the next Wiener-L\'{e}vy lemma, specialized to the Jaffard algebra $J_\alpha$, 
which follows from a more general result of Sun \cite{sun07}.

\begin{lemma}      \label{le:wiener-levy}
If $\psi\in J_\alpha$, for $\alpha>d$, and $F$ is analytic in a neighborhood of the range of 
$\psi$, then $F\circ\psi\in J_\alpha$.
\end{lemma}
%
\begin{proof}
We will use the correspondence between symbols 
in $J_\alpha$ and Laurent matrices that have off-diagonal decay of power $\alpha$. 
The relevant Wiener-L\'{e}vy result proved by Sun \cite[Theorem~5.3]{sun07} applies to 
a general Schur type algebra $\mathcal{A}_{p,w} (X,\rho, \mu)$ of infinite matrices, 
where $(X,\rho, \mu)$ refers to a so-called `space of homogeneous type', $p\in[1,\infty]$, 
and $w$ is a weight satisfying a certain `admissibility' assumption.

In our case, the index set $X:=\mathbb{Z}^d$ is equipped with the usual Euclidean distance 
$\rho$ and the usual counting measure $\mu$. Also, we need to select $p=\infty$ and the 
polynomial weight $w:=w_\alpha= (1+\|\cdot\|)^{\alpha}$, for $\alpha>d$. For these choices, 
the Schur algebra $\mathcal{A}_{p,w}$ coincides, in the notation of \cite{sj90}, with 
$Q_{\alpha}$, the Jaffard algebra of matrices with polynomial off-diagonal decay of 
power $\alpha$. Note that the technical assumption that the weight $w$ should be 
$(p,2)$-admissible is satisfied by the polynomial weight for $\alpha>d$, as shown in 
\cite[Example~A.2]{sun07}.

Now, the stated Wiener-L\'{e}vy lemma for $J_\alpha$ follows from Sun's result for 
$Q_{\alpha}$, since the range of a continuous symbol $\psi$ coincides with the 
spectrum of its associated Laurent matrix $L_\psi$, and the Laurent matrix of a product 
of symbols is the product of the Laurent matrices of the symbol factors.
\end{proof}
\medskip

\noindent
\textbf{First proof of Theorem~\ref{th:SCI-decay-symb}.}
\textbf{(i)} 
It follows from the proof of Theorem~\ref{th:CI-decay-poly} that the inverse symbol 
$\omega = 1 / \sigma\in J_{\alpha}$, therefore by Lemma~\ref{le:wiener-levy} we obtain 
that $\log \omega \in J_{\alpha}$, as well. In turn, this implies that the function $\Lambda_+$ 
defined by (\ref{eq:lambda+}) or (\ref{eq:lambda++}) satisfies $\Lambda_+ \in J_{\alpha}$. 
A new application of Lemma~\ref{le:wiener-levy}, with $F:=\exp$, shows that 
$\omega_+ = \exp (\Lambda_+) \in J_{\alpha}$, hence its Fourier coefficients satisfy 
(\ref{eq:om+coeffs-poly}), as stated.

Moving to the decay of the Fourier coefficients of $\omega_j$, we first let 
$H=\mathbb{Z}^d_{\preceq,+}$. Substituting the expression (\ref{eq:twisted-trunc}) of the 
second factor of $\omega_j$ in (\ref{eq:sc-symbol}), implies
\begin{equation}
\zeta^{-j} \omega_j (\zeta) = \omega_{+} (\zeta) \omega_{-}^{(j)} (\zeta) .
\label{eq:sc-symb-fact}
\end{equation}
The left-hand side of this equation has the Fourier expansion
\begin{equation}
\zeta^{-j} \omega_j (\zeta) = \sum_{k\succeq 0} a_{k,j} \zeta^{k-j} .
\label{eq:sc-symb-shift}
\end{equation}
On the other hand, since $\omega_{-}^{(j)}$ is a truncation of $\omega_{+}(\zeta^{-1})$, the 
Fourier coefficients of both factors of the right-hand side of (\ref{eq:sc-symb-fact}) decay 
according to estimate (\ref{eq:om+coeffs-poly}). By Lemma~\ref{le:DCP}(i), it follows that 
the Fourier coefficients of the product of these two factors decay with the same power $\alpha$, 
hence (\ref{eq:sc-symb-decay-poly}) holds. In addition, due to the uniform estimate 
(\ref{eq:om+tr-bound}), the decay constant of the above product can be bounded, for all $j\in H$, 
by $C_2 := 2^{\alpha+1} C_1 \| \omega_{+} \|_W$.

Turning to the case $H = \mathbb{Z}^{d-1}\!\!\times\!\mathbb{Z}_+$, it is sufficient to replace 
$\omega_{-}^{(j)}$ with the truncation $\omega_{-}^{(j_d)}$ defined by (\ref{eq:om+trunc}).
Then all arguments of the previous paragraph apply, which establishes 
(\ref{eq:sc-symb-decay-poly}) in this case, as well.
\smallskip

\textbf{(ii)}
By reducing $\alpha$ if needed, as in the proof of Theorem~\ref{th:CI-decay-exp}, 
both the cardinal symbol $\sigma(\zeta)$ and its reciprocal $\omega(\zeta)=1/\sigma(\zeta)$, 
can be extended as analytic and sectorial functions of $d$ complex variables in the 
neighborhood $\mathcal{A}^d_\alpha$ of the unit torus $T^d$. Hence $\log \omega (\zeta)$ 
is also analytic and admits an absolutely convergent Laurent-Fourier expansion (\ref{eq:log-omega}) 
for $\zeta\in\mathcal{A}^d_\alpha$, with exponentially decaying coefficients 
$\{\lambda_k\}_{k\in\mathbb{Z}^d}$. It follows that $\Lambda_+$, as defined by 
(\ref{eq:lambda+}) or (\ref{eq:lambda++}), is also analytic on $\mathcal{A}^d_\alpha$, 
and the same property applies to $\omega_+ = \exp (\Lambda_+)$. We deduce that 
$\omega_+$ has an absolutely convergent Laurent-Fourier 
representation on $\mathcal{A}^d_\alpha$, with coefficients indexed by $H$, hence, 
by Cauchy's estimates, (\ref{eq:om+coeffs}) holds for some $\beta\in(0,\alpha)$ and 
$C_1:=\max \{ |\omega_+(\zeta)| : |\zeta_\tau|=e^{\pm\beta}, \tau = 1,\ldots,d\}$. (The arguments 
of this paragraph have previously been used by Goodman et al.\ \cite[Theorem~2.4]{gmrs00} 
in the context of Wiener-Hopf factorizations of symbols of Gram-Laurent matrices).

To obtain the decay of the semi-cardinal symbol coefficients, we first consider the case 
$H=\mathbb{Z}^d_{\preceq,+}$. As in the above proof of \textbf{(i)}, we use (\ref{eq:sc-symb-fact}) 
and (\ref{eq:sc-symb-shift}) to express the shifted semi-cardinal symbol $\zeta^{-j} \omega_j$ 
in two ways. Since $\omega_{-}^{(j)}$ is a truncation of $\omega_{+}(\zeta^{-1})$, 
the Fourier coefficients of both factors on the right-side of (\ref{eq:sc-symb-fact}) 
decay exponentially according to (\ref{eq:om+coeffs}), with constants independent of the 
truncation index $j$. Hence, we may use Lemma~\ref{le:DCP}(ii) to deduce the exponential 
decay (\ref{eq:sc-symb-decay}), with $C_2:=C_1^2  \coth^d ( \frac{\beta-\beta'}{2} )$. 

Next, consider the case $H=\mathbb{Z}^{d-1}\!\!\times\!\mathbb{Z}_+$. Then the above proof 
still applies, provided we replace the truncation $\omega_{-}^{(j)}$ with $\omega_{-}^{(j_d)}$, 
as defined in (\ref{eq:om+trunc}). However, the following arguments show that the rate $\beta'$ 
can be improved to $\beta$ in the last component. Indeed, note that, due to the 
translation-invariance properties (\ref{eq:invariance-c}) and (\ref{eq:invariance-L}), it is sufficient 
to consider the case $j'=0$. Hence, estimating the explicit expression (\ref{eq:sc-coeffs}) 
via (\ref{eq:om+coeffs}) and rearranging, we have
\begin{eqnarray*}
|a_{(0,j_d),(k',k_d)}| & \leq & \sum_{l_d=0}^{\min\{j_d,k_d\}} \sum_{l'\in\mathbb{Z}^{d-1}} 
|\gamma_{(l',j_d-l_d)}|\, |\gamma_{(k'-l',k_d-l_d)}|  \\
& \leq & C_1^2 \prod_{p=1}^{d-1}  \sum_{l_p\in\mathbb{Z}} 
e^{-\beta |l_p|} e^{-\beta |k_p - l_p|}  \\
& \mbox{ } & \quad \times \sum_{l_d=0}^{\min\{j_d,k_d\}} 
e^{-\beta |j_d - l_d|} e^{-\beta |k_d - l_d|} 
\end{eqnarray*}
Now, for each $\beta'\in(0,\beta)$ and $p=1,\ldots,d-1$, the proof of Lemma~\ref{le:DCP}(ii) 
shows that each of the above sums indexed over $l_p$ is of the magnitude $O (e^{-\beta' |k_p|})$, 
with a constant that only depends on $\beta$ and $\beta'$. For the finite sum over $l_d$, it is 
sufficient to consider the case $j_d \leq k_d$, for which 
\begin{equation*}
\sum_{l_d=0}^{j_d} e^{-\beta (j_d - l_d)} e^{-\beta (k_d - l_d)}
\leq e^{-\beta (k_d-j_d)} \sum_{n=0}^{j_d} e^{-\beta n}
< \frac{e^{-\beta (k_d-j_d)}}{1-e^{-\beta}}.
\end{equation*}
This implies the decay (\ref{eq:sc-symb-decay}), with the improved right-hand side exponential 
rate $e^{-\beta' |k'-j'|-\beta |k_d-j_d|}$, as stated. 
\ \rule{0.5em}{0.5em}
\medskip

The second proof we give for Theorem~\ref{th:SCI-decay-symb} uses an indirect operatorial 
argument, applying Jaffard's inverse-closedness results \cite{sj90} to the Toeplitz matrix $T_\phi$ 
of the semi-cardinal interpolation system. This is much shorter than the first proof, but at the same 
time much less transparent, as it does not involve the decay of Fourier coefficients of the 
Wiener-Hopf factor $\omega_+$. Also, it does not capture the decay improvement in 
(\ref{eq:sc-symb-decay}) for $H=\mathbb{Z}^{d-1}\!\!\times\!\mathbb{Z}_+$. 
\medskip

\noindent
\textbf{Second proof of Theorem~\ref{th:SCI-decay-symb}.}
\textbf{(i)} 
Condition (\ref{eq:decay-alg}) implies $\phi(j-k) = O ((1+\|j-k\|)^{-\alpha})$, for all $j,k\in H$, 
hence $T_\phi \in Q_\alpha$, the Jaffard algebra of matrices indexed by $H\times H$ with 
polynomial off-diagonal decay of power $\alpha$. By Corollary~\ref{cor:SCI-operator}, $T_\phi$ 
is also  invertible as a bounded linear operator on $\ell_2 (H)$. Then \cite[Proposition~3]{sj90} 
implies $T_\phi^{-1} \in Q_\alpha$, which is equivalent to (\ref{eq:sc-symb-decay-poly}), since 
$T_\phi^{-1} = [a_{k,j}]_{k,j\in H}$.
\smallskip

\textbf{(ii)} 
If $\phi$ decays exponentially, a similar argument applies 
to prove (\ref{eq:sc-symb-decay}), by replacing $Q_\alpha$ with the corresponding 
algebra of matrices with exponential decay and by using \cite[Proposition~2]{sj90}.
\ \rule{0.5em}{0.5em}
\medskip

\noindent
\textbf{Decay transfer to semi-cardinal Lagrange functions}
%

\begin{theorem}	\label{th:SCI-decay-Lag}
Suppose $\phi$ satisfies the hypotheses of \emph{Theorem~\ref{th:E&U-CI}}, and 
$H$ is a half-space lattice.

\textbf{\emph{(i) Algebraic decay.}} If $\phi$ has algebraic decay with power $\alpha>d$, 
i.e.\ \emph{(\ref{eq:decay-alg})} holds, then the Lagrange functions $\chi_j$ decay with the 
same power about $j$, that is, there exists a constant $C_{3}=C_{3}(\alpha,d,\phi)$, such that 
\begin{equation}
| \chi_j ( x ) | \leq  C_3 (1+\|x-j\|)^{-\alpha} 
\quad   x\in \mathbb{R}^d,\  j\in H.
\label{eq:SCI-Lagr-decay-poly}
\end{equation}

\textbf{\emph{(ii) Exponential decay.}} If $\phi$ decays exponentially with rate $\alpha>0$,   
i.e.\ \emph{(\ref{eq:decay-exp})} holds, and $\beta\in(0,\alpha)$ is the rate of decay of the 
Fourier coefficients of $\omega_+$, according to \emph{(\ref{eq:om+coeffs})}, then, 
for each $\beta'\in(0,\beta)$, there exists a constant $C_{3}=C_{3}(\beta,\beta',\phi,H,d)$, 
such that
\begin{equation}
\vert \chi_j ( x ) \vert \leq C_3 \, e^{-\beta' |x-j|} , \quad   x\in \mathbb{R}^d,\  j\in H.
\label{eq:SCI-Lagr-decay}
\end{equation}
Further, if $H=\mathbb{Z}^{d-1}\times\mathbb{Z}_{+}$, the exponential  
in the last inequality can be replaced by $e^{-\beta' |x'-j'|-\beta |x_d-j_d|}$.
\end{theorem}

\begin{proof}
\textbf{(i)}
For $j\in H$, we use the kernel representation (\ref{eq:SCI-lagrange}) of $\chi_j$, 
along with (\ref{eq:decay-alg}) and (\ref{eq:sc-symb-decay-poly}), to estimate
\begin{eqnarray*}
|\chi_j (x)| & \leq & \sum_{k\in H} |a_{k,j}|\, |\phi (x-k)|  \\
& \leq & C_2 C_0 \sum_{k\in H}  (1+\|k-j\|)^{-\alpha} (1+\|x-k\|)^{-\alpha}.
\label{eq:SCI-RHS-poly}
\end{eqnarray*}
The last sum admits the following expression as upper bound:
\begin{eqnarray*}
\sum_{k\in \mathbb{Z}^{d}}  (1+\|k-j\|)^{-\alpha} (1+\|x-k\|)^{-\alpha}  
& = & \sum_{k\in \mathbb{Z}^{d}}  (1+\|k\|)^{-\alpha} (1+\|x-j-k\|)^{-\alpha}  \\
& \leq &  C(\alpha,d)  (1+\|x-j\|)^{-\alpha},
\end{eqnarray*}
where the last inequality follows by replacing $x$ with $x-j$ in the estimate of the sum 
(\ref{eq:CI-RHS-poly}) obtained, for the cardinal case, in the proof of 
Theorem~\ref{th:CI-decay-poly}. This establishes (\ref{eq:SCI-Lagr-decay-poly}).

\textbf{(ii)} 
The corresponding estimate (\ref{eq:SCI-Lagr-decay}) follows from the decay properties 
(\ref{eq:decay-exp}), (\ref{eq:sc-symb-decay}), and the representation (\ref{eq:SCI-lagrange}):
\begin{eqnarray*}
|\chi_j (x)| & \leq & \sum_{k\in H} |a_{k,j}|\, |\phi (x-k)|  \ 
\leq\  C_2 C_0 \sum_{k\in H} e^{-\beta' |k-j|} e^{-\alpha |x-k|}  \\
& \leq & C_2 C_0 \sum_{k\in \mathbb{Z}^d} e^{-\beta' |k-j|} e^{-\alpha |x-k|} \ 
=\ C_2 C_0 \sum_{k\in \mathbb{Z}^d} e^{-\beta' |k|} e^{-\alpha |x-j-k|}  \\
& \leq & C_3\, e^{-\beta' |x-j|}, \quad   x\in \mathbb{R}^d,\  j\in H.
\end{eqnarray*}
The last inequality is a consequence of the estimate obtained for the cardinal case in the proof 
of Theorem~\ref{th:CI-decay-exp}, with $C_3:=C_2 C_0 \coth^d ( \frac{\alpha-\beta'}{2} )$. 

It is now straightforward to see how the improved decay property for the coefficients 
$a_{k,j}$ carries over to the Lagrange functions $\chi_j$, making some necessary 
changes in the last display.
\end{proof}
\medskip

\noindent
\textbf{Remark.} As in the case of cardinal interpolation, Theorems~\ref{th:SCI-decay-symb} 
and \ref{th:SCI-decay-Lag} can be adapted to kernels $\phi$ with anisotropic algebraic decay, 
based on the version of Wiener's lemma given in Gr\"{o}chenig and Klotz \cite[Theorem~1.1]{gk10} 
for matrices indexed on $H\times H$. Similarly, the weighted $\ell_1$-condition 
(\ref{eq:weighted-decay}) is transferred to the inverse Toeplitz matrix $T_\phi^{-1}$ via the 
inverse-closed weighted Schur algebra $\mathcal{A}^1_w$ treated by Gr\"{o}chenig and 
Leinert \cite{gl06} (Remark~2 after Corollary~7) and by Sun \cite[Theorem~4.1]{sun07}. 
The decay induced on the semi-cardinal Lagrange functions 
is then expressed as finiteness of a weighted Lebesgue constant: 
\begin{equation*}
\sup_{x\in\mathbb{R}^d} \sum_{j\in H} | \chi_j (x) | w(j) < \infty.
\end{equation*}
%

\section{Further properties}
\addtocounter{equation}{-33}

This section deals with two topics: first, we show that the Lagrange function for 
cardinal interpolation can be approximated with shifts of semi-cardinal Lagrange 
functions; secondly, we obtain the variational characterization of the cardinal and 
semi-cardinal Lagrange functions and their associated schemes.

\subsection{Convergence of semi-cardinal interpolation to cardinal interpolation}

Let $H$ be a half-plane lattice as in section~3, and assume that $\phi$ 
satisfies the hypotheses of Theorem~\ref{th:E&U-CI}. Since any function of the 
form (\ref{eq:SCI-repres}) satisfies the reproduction formula (\ref{eq:CI-reprod}) 
on the cardinal grid $\mathbb{Z}^d$, it follows that, for each $j\in H$, the 
Lagrange function $\chi_j$ of the corresponding semi-cardinal interpolation scheme 
on $H$ admits the representation:
\begin{equation*}
\chi_j ( x ) = \chi (x-j) + \sum_{k\in\mathbb{Z}^{d}\setminus H} \chi_j (k) \chi ( x-k ) ,
\quad x\in\mathbb{R}^d.
\end{equation*}
This identity suggests already that the difference $\chi_j - \chi(\cdot - j)$ may be 
made small as $j$ becomes suitably large. To establish such a property, we employ 
a different representation, which is a more natural tool for the task.

We start by defining, for both types of half-space lattices, the function
\begin{equation}
\eta ( x ) := \sum_{k\in H} \gamma_k\, \phi ( x-k ) ,
\quad x\in\mathbb{R}^d,
\label{eq:psi-def}
\end{equation}
where $\{\gamma_k\}_{k\in H}$ is the sequence of Fourier coefficients of the 
Wiener-Hopf factor $\omega_+$ appearing in (\ref{eq:omega+})--(\ref{eq:om+Fourier}).
Since $\phi$ is bounded and $\{\gamma_k\}\in\ell_1(H)$, we deduce that 
(\ref{eq:psi-def}) is an absolutely and uniformly convergent series, hence $\eta$ is 
continuous and bounded on $\mathbb{R}^d$. Moreover, in analogy with the 
decay properties of $\chi$ established in section~2, we see that not only 
(\ref{eq:minimal}) transfers from $\phi$ to $\eta$, but also any additional algebraic 
or exponential decay that may apply to $\phi$.
\medskip

\noindent
\textbf{Remark.} 
The above function $\eta$ bears a close relationship with the so-called 
`limiting profile' studied by Goodman et al.\ \cite{gmrs00} in the different 
context of infinite Gram matrices generated by shifts of multivariate kernels.
%

\begin{proposition}		\label{prop:psi-rep}
Let $\phi$ satisfy the hypotheses of \emph{Theorem~\ref{th:E&U-CI}}.
For both types of half-space lattices $H$, the following representation holds:
\begin{equation}
\chi (x) = \sum_{l\in H} \gamma_l\, \eta ( x+l )  , \quad x\in\mathbb{R}^d.
\label{eq:psi-rep-chi} 
\end{equation}
Further, if $H=\mathbb{Z}^d_{\preceq,+}$ and $j\in H$, 
\begin{equation}
\chi_j (x) = \sum_{0\preceq l\preceq j} \gamma_{j-l}\, \eta ( x-l ),  \quad x\in\mathbb{R}^d,
\label{eq:psi-rep-chij}
\end{equation}
while, for $H=\mathbb{Z}^{d-1}\times\mathbb{Z}_+$ and $j=(j',j_d)\in H$, we have
\begin{equation}
\chi_j (x) = \sum_{l_d = 0}^{j_d} \sum_{l'\in\mathbb{Z}^{d-1}} \gamma_{(j'-l',j_d-l_d)}\, 
\eta ( x-l ),  \quad x\in\mathbb{R}^d.
\label{eq:psi-rep-chij-bis}
\end{equation}
\end{proposition}

\begin{proof}
For any half-space lattice $H$, we extend $\{\gamma_k\}_{k\in H}$ to a sequence 
indexed on $\mathbb{Z}^{d}$ by setting $\gamma_k := 0$ for $k\not\in H$. With this 
convention, (\ref{eq:wh}) and (\ref{eq:om+Fourier}) imply
\begin{equation*}
a_k = \sum_{l\in H} \gamma_{k+l} \gamma_{l}, \quad k\in \mathbb{Z}^{d}.
\end{equation*}
Substituting $a_k$ in the kernel representation (\ref{eq:CI-lagrange}) of $\chi$ and rearranging,
\begin{eqnarray*}
\chi (x) 
& = & \sum_{k\in\mathbb{Z}^{d}} \left( \sum_{l\in H} \gamma_{k+l} \gamma_{l} \right) \phi (x-k)  \\
& = & \sum_{l\in H} \gamma_l \sum_{k+l\in H} \gamma_{k+l}\,  \phi (x-k)  \\
& = & \sum_{l\in H} \gamma_l \sum_{k\in H} \gamma_{k}\,  \phi (x+l-k) ,
\end{eqnarray*}
which provides (\ref{eq:psi-rep-chi}), since the sum of the last series is $\eta ( x+l )$. 

Next, for $H=\mathbb{Z}^d_{\preceq,+}$, identity (\ref{eq:psi-rep-chij}) is obtained by 
substituting the explicit formula (\ref{eq:sc-coeffs+}) for $a_{k,j}$ into the kernel 
representation (\ref{eq:SCI-lagrange}) of $\chi_j$:
\begin{eqnarray*}
\chi_j (x) 
& = & \sum_{k\in H} \left( \sum_{0\preceq l\preceq j} \gamma_{k-l} \gamma_{j-l} \right) \phi (x-k)  \\
& = & \sum_{0\preceq l\preceq j} \gamma_{j-l} \sum_{k\in H} \gamma_{k-l}\,  \phi (x-k)  \\
& = & \sum_{0\preceq l\preceq j} \gamma_{j-l} \sum_{k\in H} \gamma_{k}\,  \phi (x-l-k) .
\end{eqnarray*}
Also, (\ref{eq:psi-rep-chij-bis}) follows in a similar way if $H=\mathbb{Z}^{d-1}\times\mathbb{Z}_+$, 
via re-arrangements with respect to the last coordinate of the summation index.
\end{proof}

\bigskip

The next result establishes the uniform convergence of semi-cardinal Lagrange 
functions $\chi_j(\cdot + j)$ to $\chi$. Note that, if $H=\mathbb{Z}^{d-1}\times\mathbb{Z}_+$ 
and $j=(j',j_d)\in H$, then, by (\ref{eq:invariance-L}), 
$ \chi_{j} (x+j) = \chi_{(0,j_d)} (x',x_d+j_d) $, which is independent of $j'$.

\begin{theorem}      \label{th:SCI-to-CI}
Assume that $\phi$ satisfies the hypotheses of \emph{Theorem~\ref{th:E&U-CI}}.

\emph{(i)} If $H=\mathbb{Z}^d_{\preceq,+}$, define the nested sequence 
$\{H_n\}_{n=0}^\infty$ of finite sets
\begin{equation*}
H_n := \{ k\in H : \|k\|_2 \leq n \},
\end{equation*}
and, for each $n\geq 0$, let $j^{(n)}$ denote the maximum element of $H_n$ with 
respect to the order relation $\preceq$. Then,
\begin{equation}
\lim_{n\to\infty} \sup_{x\in\mathbb{R}^d} \left| \chi ( x ) - 
\chi_{j^{(n)}} \left( x+j^{(n)} \right) \right| = 0.
\label{eq:conv1}
\end{equation}

\emph{(ii)} If $H=\mathbb{Z}^{d-1}\times\mathbb{Z}_+$, then
\begin{equation}
\lim_{j_d\to\infty} \sup_{x\in\mathbb{R}^d} | \chi ( x ) - \chi_{(0,j_d)} (x',x_d+j_d) | = 0.
\label{eq:conv2}
\end{equation}
\end{theorem}
%
\begin{proof}
(i) If $H=\mathbb{Z}^d_{\preceq,+}$ and $j\in H$, expression (\ref{eq:psi-rep-chij}) implies, 
on re-indexing and comparing with (\ref{eq:psi-rep-chi}):
\begin{eqnarray}
\chi (x) - \chi_j (x+j) 
& = & \chi (x) - \sum_{0\preceq l\preceq j} \gamma_{l}\, \eta ( x+l ) \nonumber \\
& = & \sum_{l\succ j} \gamma_{l}\, \eta ( x+l ),  \quad x\in\mathbb{R}^d.
\label{eq:psi-rep-error}
\end{eqnarray}
Hence, since $\eta$ is bounded on $\mathbb{R}^d$, we obtain the estimate
\begin{equation}
| \chi (x) - \chi_j (x+j) | \leq M \sum_{l\succ j} | \gamma_{l} | ,  \quad x\in\mathbb{R}^d,
\label{eq:psi-rep-est}
\end{equation}
where $M:=\sup_{x\in\mathbb{R}^d} |\eta(x)|$. The result follows from this, noting that 
the definition of $j^{(n)}$ implies $\{ l\in\mathbb{Z}^d : l \succ j^{(n)} \} \subset H\setminus H_n$, 
and that
\begin{equation*}
\lim_{n\to\infty} \sum_{l\in H\setminus H_n} | \gamma_{l} | = 0,
\end{equation*}
since $\{\gamma_l\}_{l\in H} \in\ell_1 (H)$ and $\cup H_n = H$.

(ii) If $H=\mathbb{Z}^{d-1}\times\mathbb{Z}_+$, then (\ref{eq:psi-rep-chij-bis}) leads to a 
representation of $\chi - \chi_j(\cdot + j)$ similar to (\ref{eq:psi-rep-error}), in which the 
summation is taken over all multi-indices $l=(l',l_d)$ with $l'\in\mathbb{Z}^{d-1}$ and 
$l_d>j_d$. Since the absolute summability of $\{\gamma_l\}_{l\in H}$ implies
\begin{equation*}
\lim_{j_d\to\infty} \sum_{l_d>j_d} \sum_{l'\in\mathbb{Z}^{d-1}} | \gamma_{(l',l_d)} | = 0,
\end{equation*}
the conclusion is obtained via the estimate analogous to (\ref{eq:psi-rep-est}).
\end{proof}
\bigskip

\noindent
\textbf{Remarks.} 
(i) The above sequence $\{H_n\}_{n=0}^\infty$ may be replaced by more general nested 
sequences of finite sets with union $H$, as considered in \cite[Proposition~3.3]{vsr03}, 
where such sequences are used in connection with the projection (or finite section) 
method for semi-infinite linear systems. \smallskip

(ii) If the kernel $\phi$ satisfies algebraic or exponential decay, as in 
Theorem~\ref{th:SCI-decay-Lag}, upper bounds on the 
rate of convergence in (\ref{eq:conv1}) and (\ref{eq:conv2}) can be obtained using 
standard tail estimates of the absolutely convergent series $\sum_{k\in H} |\gamma_k|$, 
based on the corresponding decay of the coefficients $\gamma_k$. Specifically, the 
uniform error in (\ref{eq:conv2}) is of magnitude $O ((1+j_d)^{d-\alpha})$ if $\phi$ decays 
with algebraic power $\alpha>d$, and of magnitude $O(e^{-\alpha'j_d})$ if $\phi$ decays 
with exponential rate $\alpha>0$, where $\alpha'\in(0,\alpha)$. Similar rates apply to 
(\ref{eq:conv1}), replacing $j_d$ by $n$. \smallskip

(iii) Using Fourier transforms, a result similar to Theorem~\ref{th:SCI-to-CI} was proved in 
\cite[Theorem~8]{ab02}, for biharmonic splines and $H=\mathbb{Z}\!\times\!\mathbb{Z}_+$. 
For the integrable kernels $\phi$ treated here, one can also establish, as in \cite{ab02}, 
the approximation of the cardinal scheme (\ref{eq:CI-scheme}) by a sequence 
of semi-cardinal schemes of the type (\ref{eq:SCI-scheme}), suitably defined on shifted 
lattices $-k+H$, for $k\in H$. 
%

\subsection{Native space variational property}

The final result of the paper assumes that, in addition to the hypotheses of 
Theorem~\ref{th:E&U-CI}, the kernel $\phi$ is a positive definite function. In this case, 
it is known that its nonnegative Fourier transform $\widehat{\phi}$ belongs to 
$L_1 (\mathbb{R}^d)$; see Wendland \cite[Corollary~6.12]{wen05}. Further, by 
\cite[Theorem~10.12]{wen05}, the translation invariant kernel $\phi(\cdot - \cdot)$ is the 
reproducing kernel of the `native space' $\mathcal{N}_{\phi}$ identified with the real 
Hilbert space of functions $f\in L_2 (\mathbb{R}^d) \cap C (\mathbb{R}^d)$ for which 
there exists, say, $\mathcal{Q} f \in L_2 (\mathbb{R}^d)$, such that
\begin{equation}
\widehat{f} = (\widehat{\phi})^{1/2} \mathcal{Q} f .
\label{eq:native-rep}
\end{equation}
The inner product on $\mathcal{N}_{\phi}$ is then defined by
\begin{equation*}
(f,g)_{\phi} = \frac{1}{(2\pi)^d} \int_{\mathrm{supp}\,\widehat{\phi}}
\mathcal{Q} f (t)\, \overline{\mathcal{Q} g (t)} \, dt.
\end{equation*}
%

\begin{theorem}      \label{prop:CI-var}
Let $\phi$ be a positive definite function satisfying the hypotheses of 
\emph{Theorem~\ref{th:E&U-CI}}.

\emph{(i)} The corresponding Lagrange function $\chi$ for cardinal interpolation belongs 
to the native space $\mathcal{N}_{\phi}$ and the following `fundamental identity' holds:
\begin{equation}
(f,\chi)_{\phi} = \sum_{k\in\mathbb{Z}^d} a_k f(k), \quad  f\in\mathcal{N}_{\phi},
\label{eq:CI-fund-id}
\end{equation}
where $\{a_k\}_{k\in\mathbb{Z}^d}$ are the coefficients of the kernel representation 
\emph{(\ref{eq:CI-lagrange})} of $\chi$. Consequently, $\chi$ is the unique function of minimum-norm 
in $\mathcal{N}_{\phi}$ that satisfies the interpolation conditions \emph{(\ref{eq:CI-delta})}.

\emph{(ii)} If $H$ is a half-space lattice in $\mathbb{R}^d$, then, for each $j\in H$, the 
corresponding Lagrange function $\chi_j$ for semi-cardinal interpolation belongs to the native 
space $\mathcal{N}_{\phi}$ and the following `fundamental identity' holds:
\begin{equation*}
(f,\chi_j)_{\phi} = \sum_{k\in H} a_{k,j} f(k), \quad f\in\mathcal{N}_{\phi},
\end{equation*}
where $\{a_{k,j}\}_{k\in H}$ are the coefficients of the kernel representation 
\emph{(\ref{eq:SCI-lagrange})} of $\chi_j$. Hence, $\chi_j$ is the unique function of minimum-norm 
in $\mathcal{N}_{\phi}$ that satisfies the interpolation conditions \emph{(\ref{eq:SCI-delta})}.
\end{theorem}

\begin{proof}
(i) The continuity of $\chi$ follows from the uniform convergence of its kernel representation 
(\ref{eq:CI-lagrange}). Further, it was shown in \cite[section~2]{jm91} that assumption 
(\ref{eq:minimal}) implies $\phi\in L_p (\mathbb{R}^d)$, for any $p\in[1,\infty]$. In particular, 
$\phi\in L_2 (\mathbb{R}^d)$ implies that (\ref{eq:CI-lagrange}) is an absolutely convergent 
series in $L_2 (\mathbb{R}^d)$, since $a\in\ell_1$ and $\| \chi (\cdot - k) \|_{L_2} = \| \chi \|_{L_2}$, 
for all $k\in\mathbb{Z}^d$. Therefore, the pointwise sum $\chi$ of 
this series is equal almost everywhere to its $L_2$-sum, hence $\chi\in L_2 (\mathbb{R}^d)$.

Similarly, $\phi\in L_1 (\mathbb{R}^d)$ implies that (\ref{eq:CI-lagrange}) is an absolutely 
convergent series in $L_1 (\mathbb{R}^d)$, and we may apply the Fourier transform termwise, 
to obtain
\begin{equation}
\widehat{\chi} (t) = \sum_{k\in \mathbb{Z}^d} a_k [\phi (\cdot-k)]^{\wedge} (t)
= \widehat{\phi} (t) \sum_{k\in \mathbb{Z}^d} a_k e^{-ikt} = \widehat{\phi} (t) \, \omega (t) ,
\quad t\in\mathbb{R}^d,
\label{eq:CI-FT}
\end{equation}
where $\omega=1/\sigma_\phi$. Since the symbol $\omega$ is bounded and 
$\widehat{\phi} \in L_1 (\mathbb{R}^d)$, we can choose
$\mathcal{Q}\chi := (\widehat{\phi})^{1/2} \omega \in L_2 (\mathbb{R}^d)$, hence (\ref{eq:CI-FT}) 
implies $\chi\in\mathcal{N}_{\phi}$.

Next, for any $f\in\mathcal{N}_{\phi}$, the factorization (\ref{eq:native-rep}) implies 
$\widehat{f} \in L_1 (\mathbb{R}^d)$. This fact, together with $f\in C (\mathbb{R}^d)$ and the 
Fourier inversion formula $f=(\widehat{f}\,)^\vee$ valid for $f\in L_2 (\mathbb{R}^d)$ 
(see Stein and Weiss \cite[Theorem~2.4]{sw71}), imply that the inversion formula holds 
everywhere as an identity between continuous functions:
\begin{equation*}
f(x) = \frac{1}{(2\pi)^d} \int_{\mathbb{R}^d} \widehat{f} (t) \, e^{ixt} \, dt,\quad x\in\mathbb{R}^d.
\end{equation*}
In particular, $f$ is bounded on $\mathbb{R}^d$. Since $\{a_k\}_{k\in\mathbb{Z}^d}\in\ell_1$, 
it follows that the series in (\ref{eq:CI-fund-id}) is absolutely convergent and we have
\begin{eqnarray*}
(2\pi)^d \, (f,\chi)_{\phi} 
& = & \int_{\mathrm{supp}\,\widehat{\phi}} \mathcal{Q}f\! (t)\, \overline{\mathcal{Q} \chi (t)} \, dt
\ = \ \int_{\mathrm{supp}\,\widehat{\phi}} \mathcal{Q}f\! (t)\, (\widehat{\phi})^{1/2}\! (t)\, 
\overline{\omega (t)} \, dt   \\
& = & \int_{\mathbb{R}^d} \widehat{f} (t) \, \overline{\omega (t)} \, dt 
\ = \ \sum_{k\in \mathbb{Z}^d} a_k  \int_{\mathbb{R}^d}  \widehat{f}(t)\, e^{ikt} \,dt  \\
& = & (2\pi)^d \sum_{k\in\mathbb{Z}^d} a_k f(k),
\end{eqnarray*}
by permitted interchange of the integration and summation operators in the fourth equality.
The identity (\ref{eq:CI-fund-id}) implies that $\chi$ is orthogonal on any $f\in\mathcal{N}_{\phi}$, 
such that $f(k) = 0$ for all $k\in\mathbb{Z}^d$. A standard Pythagorean-type argument then 
establishes the last assertion of part (i).

(ii) This follows as in the cardinal case (i), by replacing (\ref{eq:CI-FT}) with:
\begin{equation*}
\widehat{\chi_j} (t) = \widehat{\phi} (t)\, \omega_j (-t),\quad t\in\mathbb{R}^d,
\end{equation*}
which is the Fourier transform of $\chi_j$.
\end{proof}
\bigskip

\noindent
\textbf{Remark.} 
A similar proof can be used to show that, for a data sequence $y\in\ell_1$, the cardinal 
interpolant $s=I_\phi y$ given by (\ref{eq:CI-repres}) belongs to the native space 
$\mathcal{N}_{\phi}$, it satisfies a fundamental identity similar to (\ref{eq:CI-fund-id}) 
(with coefficients $\{c_k\}$ replacing $\{a_k\}$), and it is the unique norm-minimizer among 
all functions in $\mathcal{N}_{\phi}$ that satisfy the interpolation conditions (\ref{eq:CI-problem}). 
The result also holds for the semi-cardinal interpolant $s=I^H_\phi y$ of (\ref{eq:SCI-repres}) 
in place of $\chi_j$, if the data sequence $y\in\ell_1 (H)$.
%

\section{Examples}
\addtocounter{equation}{-11}

In this section, we specialize our main results to five classes of integrable 
kernels that have been studied in the cardinal interpolation literature.

\begin{example}[Gaussian] 	\label{ex:Gaussian} 
\textup{For a parameter $c>0$, the Gaussian kernel 
\begin{equation*}
\phi (x) = \exp (-c\|x\|^2), \quad x\in\mathbb{R}^d,
\end{equation*}
obviously verifies the exponential decay condition (\ref{eq:decay-exp}). Since 
\begin{equation*}
\widehat{\phi} (t) = \frac{1}{(4\pi c)^{d/2}} \exp \left( -\frac{\|t\|^2}{4c} \right) > 0, 
\quad t\in\mathbb{R}^d,
\end{equation*}
$\phi$ is positive definite and an application of the Poisson Summation Formula 
shows that the generated cardinal symbol satisfies
\begin{equation*}
\sigma_\phi (e^{it}) = \sum_{l\in\mathbb{Z}^d} \widehat{\phi} (t+2\pi l) > 0.
\end{equation*}
Cardinal interpolation on $\mathbb{Z}^d$ with the Gaussian kernel has been thoroughly 
studied by Riemenschneider and Sivakumar \cite{rs99,rs99b}. However, semi-cardinal 
interpolation has so far been considered only in the univariate case by Baxter and Sivakumar 
\cite[Theorem~2.9]{basi96}, who proved that, for $\tau\in\mathbb{R}$, the shifted 
Gaussian kernel $\phi(\cdot+\tau)$ generates an invertible semi-infinite Toeplitz matrix on 
$\ell_2 (\mathbb{Z}_+)$ if and only if $|\tau|<1/2$.
}

\textup{Our new results in this paper establish the main properties of semi-cardinal interpolation 
with the Gaussian kernel on any half-space lattice of $\mathbb{Z}^d$, including the exponential 
decay of the associated Lagrange functions and of their kernel expansion coefficients 
(Theorems~\ref{th:SCI-decay-symb} and \ref{th:SCI-decay-Lag}). 
}
\end{example}

\begin{example}[Mat\'{e}rn] 	\label{ex:Matern}
\textup{The Mat\'{e}rn kernel is given by 
\begin{equation*}
\phi (x) = \|x\|^{m-\frac{d}{2}} K_{m-\frac{d}{2}} (\|x\|), \quad x\in\mathbb{R}^d, 
\end{equation*}
where $m\in\mathbb{R}$, $m>\frac{d}{2}$, and $K_{m-\frac{d}{2}}$ is a modified Bessel 
function. Since, for $\nu>0$, $K_{\nu} (r)$ is continuous on $(0,\infty)$ and $r^{\nu} K_{\nu} (r)$ 
can be extended by continuity at $r=0$, it follows that $\phi$ is continuous on $\mathbb{R}^d$. 
Also, the asymptotic behavior of $K_{\nu}$ for large argument \cite[Lemma~5.13]{wen05} 
implies $\phi (x) = O ( \|x\|^{m-\frac{d+1}{2}} e^{ -\|x\| } )$, as $\|x\|\to\infty$, so  
the exponential decay condition (\ref{eq:decay-exp}) holds, for some $\alpha\in(0,1)$.
}
 
 \textup{By \cite[Theorem~ 6.13]{wen05}, we have $\widehat{\phi} (t) = \rho_{m,d} (1+\|t\|^2)^{-m}$, 
$t\in\mathbb{R}^d$, for some $\rho_{m,d} > 0$, hence $\phi$ is positive definite and the 
native space it generates is the Bessel potential space 
\begin{equation*}
H_2^m (\mathbb{R}^d) = \{ f\in L_2 (\mathbb{R}^d) : \widehat{f} 
= (1+\|\!\cdot\!\|^2)^{-m/2}\,\widehat{g},\ g\in L_2 (\mathbb{R}^d) \}.
\end{equation*}
This coincides with the usual Sobolev space when $m$ is a positive integer; in this case, 
Mat\'{e}rn kernels are also known under the name `Sobolev splines'. Also, $\phi$ is the 
fundamental solution of the (pseudo)differential operator $(1-\Delta)^m$, where $\Delta$ 
is the Laplace operator in $\mathbb{R}^d$. 
}

\textup{Since the condition $2m>d$ permits the application of Poisson's Summation Formula, as 
in Example~\ref{ex:Gaussian}, the associated cardinal symbol $\sigma_\phi$ is seen to satisfy 
the positivity condition (\ref{eq:pos-symbol}). Therefore Theorems~\ref{th:CI-decay-exp} and 
\ref{th:SCI-decay-Lag} provide the exponential decay of the Lagrange functions for cardinal 
and semi-cardinal interpolation with any Mat\'{e}rn kernel $\phi$ satisfying $2m>d$.
}

\textup{Note that, so far, only the univariate case ($d=1$) in which $m$ is a positive integer has 
been covered before  in the cardinal interpolation literature. In this case, since the modified 
Bessel function of half-integer order is expressible in finite terms \cite[Eq.\ 3$\cdot$71(12)]{wat66}, 
the Mat\'{e}rn kernel takes the elementary form:
\begin{equation}
\phi (x) = |x|^{m-\frac{1}{2}} K_{m-\frac{1}{2}} (|x|) = e^{ -|x| } \sum_{r=0}^{m-1} c_{m,r} |x|^r,
\label{eq:exp-L-spline}
\end{equation}
for some positive constants $\{c_{m,r}\}$. The main properties of cardinal interpolation with 
this exponential spline kernel on the real line follow from Micchelli's general `cardinal $L$-spline' 
theory \cite{micch76}. Using different methods, Micchelli's comprehensive treatment includes, 
among other results, the exponential decay of the Lagrange function $\chi$ for cardinal 
interpolation on $\mathbb{Z}$.
}

\textup{Also, in \cite{bkr07} and \cite{ab08}, cardinal and semi-cardinal interpolation schemes for a 
scaled version of the univariate kernel (\ref{eq:exp-L-spline}) have played a prominent role in 
the construction, via separation of variables, of multivariate `polysplines' interpolating continuous 
functions prescribed on equi-spaced parallel hyperplanes. These two references employ Fourier 
transforms rather than kernel representations in order to obtain the exponential decay of the 
corresponding Lagrange functions.
}

\textup{For any dimension $d$, the choice $m=\frac{d+1}{2}$ provides the particularly simple 
kernel $\phi (x) = \|x\|^{\frac{1}{2}} K_{\frac{1}{2}} (\|x\|) = \sqrt{\frac{\pi}{2}} \, e^{ -\|x\| }$. 
}
\end{example}

\begin{example}[Generalized inverse multiquadric]	\label{ex:GIM}
\textup{This kernel is given by 
\begin{equation*}
\phi (x) = (c^2+\|x\|^2)^{-m}, \quad x\in\mathbb{R}^d, 
\end{equation*}
where $2m > d$, $c>0$, so the algebraic decay condition (\ref{eq:decay-alg}) holds with 
the rate $2m$. This example essentially exchanges $\phi$ with $\widehat{\phi}$ from 
Example~\ref{ex:Matern}, and by usual transform laws we have:
\begin{equation*}
\widehat{\phi} (t) = \frac{1}{\rho_{m,d}}  \left( \frac{\|t\|}{c} \right)^{m-\frac{d}{2}} 
K_{m-\frac{d}{2}} (\|c\,t\|), \quad t\in\mathbb{R}^d.
\end{equation*}
Since $r^{\nu} K_{\nu} (r)$ takes only positive values for $\nu>0$ and $r\in[0,\infty)$, 
it follows that $\phi$ is positive definite and we may again conclude, via Poisson's 
Summation Formula, that the associated cardinal symbol satisfies the positivity condition 
(\ref{eq:pos-symbol}). Consequently, for cardinal and semi-cardinal interpolation with the 
kernel $\phi$ satisfying $2m>d$, Theorems~\ref{th:CI-decay-poly}, \ref{th:SCI-decay-symb} 
and \ref{th:SCI-decay-Lag} provide the transfer of the algebraic decay rate $2m$ 
to the Lagrange functions, as well as to their kernel expansion coefficients.
}

\textup{In the case of cardinal interpolation, this result represents a substantial improvement 
over decay rates recently obtained by Hamm and Ledford in \cite{hl16,hl18}, who derived 
their results from estimates based on the smoothness of Fourier transforms. 
Specifically, assuming $d=1$ and $m>1$ (with our notation), \cite[Corollary~2]{hl16} states 
the algebraic rate $\lceil 2m-2 \rceil$ for the decay of the Lagrange function $\chi$ for 
cardinal interpolation. Also, under the assumption $2m>2d+1$ (with our notation), 
\cite[Theorem~3.1]{hl18} proves the algebraic rate $\lfloor 2m-d \rfloor$ if 
$m\not\in\mathbb{Z}$, and $2m-d-1$ if $m\in\mathbb{Z}$, for the kernel expansion 
coefficients of $\chi$, while \cite[Corollary~4.7 \& Eq.\ (10)]{hl18} imply the 
algebraic decay rate $d+1$ for $\chi$.
}
\end{example}

\begin{example}[B-splines and box-splines] 	\label{ex:box}
\textup{The study of univariate cardinal interpolation with the central B-spline $\phi=M_n$ 
of polynomial degree $n-1$, compactly supported on $[-n/2,n/2]$, and of class 
$C^{n-2}(\mathbb{R})$, was initiated in Schoenberg's 1946 paper 
\cite{schoen46}. As pointed in \cite{schoen46}, the explicit piecewise 
polynomial form of $M_n$, as well as its Fourier transform 
\begin{equation*}
\widehat{M_n} (t) = \left(\frac{\sin(t/2)}{t/2}\right)^n,\quad t\in \mathbb{R},
\end{equation*}
can be traced back to Laplace's work on probability. 
Clearly, $M_n$ is positive definite iff $n$ is an even integer.
}

\textup{Schoenberg proved that the cardinal symbol $\sigma_\phi$ (expressed in terms 
of the so-called Euler-Frobenius polynomial) associated to $\phi=M_n$ 
satisfies the positivity condition (\ref{eq:pos-symbol}), irrespective of the parity of $n$. 
This enabled his construction (see \cite[\S4.5]{schoen73}) of \emph{cardinal} interpolation 
for the B-spline kernel $M_n$, with an exponentially decaying Lagrange function $\chi$. 
}

\textup{Subsequently, in \cite{schoen73,schoen73b}, Schoenberg employed two methods
to construct \emph{semi-cardinal} schemes for interpolation on $\mathbb{Z}_+$, 
with splines of odd degree only. The existence and uniqueness properties of these schemes 
required the introduction of boundary conditions imposed at the origin. However, Schoenberg  
obtained the exponential decay of the corresponding semi-cardinal Lagrange functions without 
relying on B-spline representations in terms of the kernel $\phi=M_n$. For arbitrary (even or odd) 
$n$, the construction of semi-cardinal interpolation based on such representations
is given in the forthcoming paper \cite{ab20}, in which boundary conditions are expressed 
in terms of finite differences (FD) of B-spline coefficients. 
}

\textup{To illustrate this approach, we briefly consider the cubic spline case $n=4$, which also 
serves as model for a bivariate extension described further below.
Since the support of $M_4$ is $[-2,2]$, any cubic spline defined on $[0,\infty)$, 
with knots at the nonnegative integers, admits a B-spline representation
\begin{equation}
s(x) = \sum_{k=-1}^\infty c_k M_4 (x-k), \quad x\in [0,\infty),
\label{eq:B-spline-rep}
\end{equation}
for some coefficients $c_k$, $k\geq-1$. After imposing the interpolation conditions 
$s(j) = y_j$, $j\in\mathbb{Z}_+$, there is still one `degree of freedom' to be utilized by a
boundary condition for $s$. The `natural' condition $s'' (0) = 0$ used in \cite{ab06} 
is equivalent to the second order FD condition 
\begin{equation}
c_{-1} - 2c_0 + c_1 = 0,
\label{eq:NAT}
\end{equation}
while the `not-a-knot' condition $s'''(1-) = s'''(1+)$, which eliminates the spline knot at $1$, 
can be expressed equivalently (see \cite{sb03}) as:
\begin{equation}
c_{-1} - 4c_0 + 6c_1 - 4c_2 + c_3 = 0.
\label{eq:NAK}
\end{equation}
In addition, for $\phi = M_4$ and $H = \mathbb{Z}_+$, representation (\ref{eq:SCI-repres}) 
of the present paper is seen to be obtained by imposing the boundary condition 
$c_{-1} = 0$ in (\ref{eq:B-spline-rep}). 
}

\textup{The extension of the \emph{cardinal} spline interpolation theory to several variables, in a 
non-tensor product setting, was started by de Boor, H\"{o}llig, and Riemenschneider \cite{dBHR85}, 
for certain bivariate `box spline' kernels, which are compactly supported and piecewise polynomial 
over a three-direction partition of the plane. The ensuing theory of cardinal interpolation with box 
splines in arbitrary dimension is much more intricate than its univariate version; for a thorough 
introduction to the main results, see the monograph \cite[Chapter~IV]{dBHR93}.
}

\textup{In \cite{ab06}, the author presented a semi-cardinal scheme for interpolation at the points 
of the half-plane lattice $H=\mathbb{Z}\!\times\!\mathbb{Z}_+$ with the three-direction box-spline 
$\phi = M_{2,2,2}\in C^2 (\mathbb{R}^d)$, whose direction matrix has every multiplicity $2$. 
This box-spline is an example of a non-radial positive definite kernel which may be regarded as 
a bivariate analog of $M_4$, and its only non-zero values on $\mathbb{Z}^2$ are:
\begin{equation*}
\phi(0,0) = \frac{1}{2}, \ \ \phi(\pm(1,1)) = \phi(0,\pm1) = \phi(\pm1,0) = \frac{1}{12}.
\end{equation*}
It follows that the associated cardinal symbol satisfies (\ref{eq:pos-symbol}), since
\begin{equation*}
\sigma_\phi (e^{it_1},e^{it_2}) = \frac{1}{6} (3 + \cos t_1 + \cos t_2 + \cos (t_1+t_2) ) \geq \frac{1}{4} ,
\quad t\in \mathbb{R}^2.
\end{equation*}
The approach of \cite{ab06} is based on the fact that, due to the size of the support of 
$M_{2,2,2}$, the restriction of a bivariate cardinal series 
$s = \sum_{k\in\mathbb{Z}^2} c_k M_{2,2,2} (\cdot-k)$ to the upper half-plane 
$\mathbb{R}\!\times\![0,\infty)$ has a representation analogous to (\ref{eq:B-spline-rep}):
\begin{equation}
s(x) = \sum_{k_2=-1}^\infty \sum_{k_1=-\infty}^\infty 
c_{k_1,k_2} M_{2,2,2} (x_1-k_1,x_2-k_2), \quad x\in \mathbb{R}\!\times\![0,\infty).
\label{eq:box-spline-rep}
\end{equation}
This permitted the formulation of a system of FD edge conditions in terms  
of the box-spline coefficients $c_{k_1,k_2}$, with the weights stencil (A) shown below, 
as the analog of the `natural' univariate stencil (\ref{eq:NAT}).
Further, Sabin and Bejancu \cite{sb03,bs05} employed the stencil (B), which extends the 
`not-a-knot' univariate stencil (\ref{eq:NAK}). 
\begin{equation*}
\textup{(A)} \quad
\begin{array}{rr}
\mbox{\ } &  		  1 \\
	   -1 &    		 -1 \\
	    1 &     \mbox{\ }  
\end{array}
\qquad\qquad\qquad
\textup{(B)} \quad
\begin{array}{rrr}
\mbox{\ } &    \mbox{\ } &  		  1 \\
\mbox{\ } &   	       -2 &    	 -2 \\
	    1 &   	        4 &    	  1 \\
  	   -2 &   	       -2 &  \mbox{\ } \\ 
	    1 &   \mbox{\ } &  \mbox{\ }  
\end{array}
\end{equation*}
(Note that the bottom weights of both stencils are applied to the coefficients $c_{k_1,k_2}$ 
corresponding to $k_2 = -1$.)
Moreover, for $\phi = M_{2,2,2}$ and $H=\mathbb{Z}\!\times\!\mathbb{Z}_+$, 
representation (\ref{eq:SCI-repres}) of the present paper can also be 
derived from (\ref{eq:box-spline-rep}) by simply imposing the zero-order FD conditions 
\begin{equation}
c_{k_1,k_2}=0,\quad  (k_1,k_2)\in\mathbb{Z}\!\times\!\{-1\}.
\label{eq:NULL}
\end{equation}
}

\textup{Using explicit expressions for box spline coefficients, \cite{ab06} and \cite{bs05} 
established not only the exponential decay of the bivariate Lagrange functions, 
but also the polynomial reproduction properties of the semi-cardinal box-spline schemes. 
In turn, these implied that the `stationary' scaled scheme with box-spline edge conditions 
based on (A) achieves exactly half of the maximal approximation order $4$ 
known to hold for scaled cardinal interpolation with $M_{2,2,2}$, 
while the edge conditions of type (B) have the effect of restoring the 
full approximation order. 
}

\textup{On the other hand, it can be proved that the stationary scaled scheme based on 
representation (\ref{eq:SCI-repres}) of this paper, i.e.\ on the boundary conditions (\ref{eq:NULL}), 
possesses no approximation order for $\phi = M_{2,2,2}$. However, the advantage of this 
scheme over the related schemes in \cite{ab06,bs05} rests on the fact that it extends, 
for any dimension and any half-space lattice $H$, to arbitrary box-spline kernels 
that generate `correct' cardinal interpolation. A similar extension for multivariable semi-cardinal 
schemes with FD boundary conditions of positive order poses a challenging problem. 
}
\end{example}

\begin{example}[Polyharmonic B-splines] 		\label{ex:psiBspline}
\textup{This example is related to cardinal interpolation with kernels $\phi(x)$ that actually grow 
as $\|x\|\to\infty$, in which case the symbol $\sigma$ cannot be defined classically and the 
associated Laurent operator $L_\phi$ is not bounded anymore on $\ell_2$.
Nevertheless, a distributional Fourier transform approach still allows the construction 
of a suitably decaying Lagrange function $\chi$ with an absolutely convergent kernel 
representation (\ref{eq:CI-lagrange}), as demonstrated, for polyharmonic kernels, 
by Madych and Nelson \cite{mn90}, and, for other non-decaying kernels 
including the multiquadrics, by Buhmann \cite{mdb90}. 
}

\textup{In particular, for each integer $m>d/2$, \cite{mn90} considered cardinal interpolation 
with the $m$-harmonic kernel
\begin{equation*}
\phi(x) = c_{m,d} \left\{
\begin{array}{ll}
\|x\|^{2m-d} \ln \|x\|, & \mathrm{if}\ d\ \mathrm{is\ even}, \\
\|x\|^{2m-d},  & \mathrm{if}\ d\ \mathrm{is\ odd}, 
\end{array}
\right.
\quad x\in\mathbb{R}^d.
\end{equation*}
For a suitable constant $c_{m,d}$, this satisfies $\Delta^m \phi = \delta$, where 
$\Delta$ is the $d$-dimensional Laplace operator and $\delta$, the Dirac distribution.
Based on the analyticity properties of the Fourier transform $\widehat{\chi}$,
Madych and Nelson proved that the corresponding Lagrange function $\chi$ and its 
coefficients of representation (\ref{eq:CI-lagrange}) decay exponentially. 
}

\textup{For each $i=1,\ldots,d$, let $\nabla_i$ denote the second-order central difference 
operator of unit step in the $i$-th variable. In analogy with odd-degree polynomial B-splines, 
Rabut \cite{rabut90,rabut92a} considered the `elementary $m$-harmonic cardinal B-spline' 
defined by 
\begin{equation*}
\psi := \left( \sum_{i=1}^d \nabla_i \right)^m \phi.
\end{equation*}
In the thin plate spline case $m=d=2$, this bell-shaped kernel had been used before by 
Dyn and Levin \cite{dl81,dl83}, for surface fitting in a bounded domain. 
Although $\psi$ does not have a compact support, 
it decays algebraically, with $\psi(x) = O (\|x\|^{-(d+2)})$, 
for large $\|x\|$. Rabut showed that the cardinal interpolation 
Lagrange function $\chi$ generated by $\phi$ also admits a series representation 
of the type (\ref{eq:CI-lagrange}) with $\phi$ replaced by $\psi$ and with 
coefficients decaying at the same algebraic rate as $\psi$. 
}

\textup{In \cite{rabut90,rabut92b}, Rabut has replaced $\sum_{i=1}^d \nabla_i$ in the 
above definition of $\psi$ with a higher order discretization of the Laplace operator. 
This led him to `high level' $m$-harmonic cardinal B-splines with faster decay, 
reaching $O (\|x\|^{-(d+2m)})$, for large $\|x\|$. Also, Van De Ville et al.\ \cite{vbu05} 
introduced a related class of polyharmonic B-splines in two and three variables, 
using a more isotropic discretization of the Laplacean. Compared to Rabut's  
`elementary' B-splines, these possess some improved features, including a 
$O (\|x\|^{-(d+4)})$ decay. Both the `high level' and the `isotropic' polyharmonic 
B-splines provide corresponding kernel representations for Madych and Nelson's 
cardinal Lagrange function $\chi$. 
}

\textup{For each choice of a polyharmonic B-spline kernel $\psi$ described above, 
our Theorem~\ref{th:SCI-decay-Lag} establishes the algebraic decay of the 
Lagrange functions for semi-cardinal interpolation on a half-space lattice, 
with the same rate as that of $\psi$. It should be remarked that, for 
$H=\mathbb{Z}^{d-1}\!\times\!\mathbb{Z}_+$, the semi-cardinal scheme generated 
here by such a kernel $\psi$ does not coincide with that constructed by Bejancu 
\cite{ab00b} via the distributional Fourier transform of $\phi$, due to different  
Wiener-Hopf factorizations of the corresponding inverse cardinal symbols. 
}

\textup{The idea of using a pseudo B-spline kernel $\psi$ generated from differences of 
polyharmonic  or of other non-decaying kernels, such as the multiquadrics, has been a constant 
theme of  the radial basis function literature, e.g., \cite{bd96,blm11,bls06,cjw92,dl81,dl83}. 
In this context, it may be of interest to ask whether, in analogy to Example~\ref{ex:box}, the 
formulation of FD boundary conditions in terms of kernel coefficients can lead 
to improved practical schemes for interpolation in bounded domains. 
}
\end{example}

\noindent
\textbf{Acknowledgements.} 
I am grateful to Thomas Hangelbroek for a useful discussion in Arcachon in June 2018 on 
one of the examples, which subsequently triggered the generalization obtained in this paper. 
Also, I thank Malcolm Sabin, who suggested the possibility of using the zero-order FD 
conditions (\ref{eq:NULL}) at the time of writing our joint papers \cite{bs05,sb03}.

\end{document}